\documentclass[12pt,reqno]{amsart}
\usepackage{amsmath, amssymb, amsthm}
\usepackage{url}
\usepackage[breaklinks]{hyperref}
\usepackage{tikz}
\usetikzlibrary{decorations.markings}

\renewcommand{\Re}{\operatorname{Re}}

\renewcommand{\(}{\left\(}
\renewcommand{\)}{\right\)}
\renewcommand{\[}{\left\[}
\renewcommand{\]}{\right\]}

\numberwithin{equation}{section}
 \theoremstyle{plain}
\newtheorem{theorem}{Theorem}[section]
\newtheorem{lemma}[theorem]{Lemma}
\newtheorem{remark}[]{Remark}

\newtheorem{conjecture}[theorem]{Conjecture}

\newtheorem{corollary}[theorem]{Corollary}

   \makeatletter
\def\proof{\@ifnextchar[{\@oproof}{\@nproof}}
\def\@oproof[#1][#2]{\trivlist\item[\hskip\labelsep\textit{#2 Proof of\
#1.}~]\ignorespaces}
\def\@nproof{\trivlist\item[\hskip\labelsep\textit{Proof.}~]\ignorespaces}

\makeatother

\begin{document}
\title[Rademacher-type formula and Tur\'{a}n inequalities for cubic overpartitions]{Rademacher-type exact formula and higher order Tur\'{a}n inequalities for cubic overpartitions}

\author{Archit Agarwal}
\address{Archit Agarwal\\ Department of Mathematics \\
Indian Institute of Technology Indore \\
Simrol,  Indore,  Madhya Pradesh 453552, India.} 
\email{archit.agrw@gmail.com,   phd2001241002@iiti.ac.in }

\author{Meghali Garg}
\address{Meghali Garg \\ Department of Mathematics \\
Indian Institute of Technology Indore \\
Simrol,  Indore,  Madhya Pradesh 453552, India.} 
\email{meghaligarg.2216@gmail.com,   phd2001241005@iiti.ac.in}

 \author{Bibekananda Maji}
\address{Bibekananda Maji\\ Department of Mathematics \\
Indian Institute of Technology Indore \\
Simrol,  Indore,  Madhya Pradesh 453552, India.} 
\email{bibek10iitb@gmail.com,  bibekanandamaji@iiti.ac.in}


\maketitle

\begin{abstract}
In 1918,  Hardy and Ramanujan made a breakthrough by developing the circle method to deduce an asymptotic formula for the partition function $p(n)$, which was later refined by Rademacher in 1937 to produce an absolutely convergent series representation for $p(n)$.  Since then,  Rademacher-type exact formulas for various partition functions have been investigated by many mathematicians. The concept of overpartitions was introduced by Lovejoy and Corteel in $2004$.  Kim,  in $2010$,   studied an overpartition analogue of cubic partitions,  termed as cubic overpartitions. The main objective of this paper is to establish a Rademacher-type exact formula for cubic overpartitions and, as an application, to derive an explicit error term that leads to their log-concavity.  Furthermore, applying a result of Griffin, Ono, Rolen, and Zagier, we establish higher-order Turán inequalities for cubic overpartitions.   In addition, we obtain log-subadditivity and generalized log-concavity properties for cubic overpartitions  inspired by the work of Bessenrodt–Ono and DeSalvo–Pak on the ordinary partition function.
\end{abstract}

\tableofcontents

\section{Introduction}

The theory of integer partitions has captivated attention of numerous mathematicians due to its rich combinatorial structure and deep connections with analysis, number theory, and modular forms. A partition of a non-negative integer $n$ is a way of writing $n$ as a sum of positive integers, where the order of summands is irrelevant. The number of such partitions is denoted by $p(n)$. The generating function for $p(n)$ was first given by  Euler \cite{Euler} way back in 1748, 
\begin{align*}
f(q):=\sum_{n=0}^\infty p(n) q^n = \frac{1}{(q;q)_\infty}.
\end{align*}
 Here and throughout this paper, $q$ is a complex number with $|q|<1$ and $(a;q)_\infty$ denotes the $q$-Pochhammer symbol defined as
$
(a;q)_\infty := \prod_{k=0}^{\infty} (1 - aq^k).
$

At first glance, determining the value of $p(n)$ might seem a straightforward combinatorial exercise. However, proving deep results about the partition function involves techniques from complex analysis, modular forms as well as knowledge of Kloosterman sums and Bessel functions. A major breakthrough in the study of partition function occurred in 1918, when Hardy and Ramanujan \cite{HR1918} developed the circle method to establish the following asymptotic formula for $p(n)$:
\begin{align}\label{asymptotic formula for p(n)}
p(n) \sim \frac{1}{4n\sqrt{3}} \exp\left( \pi \sqrt{\frac{2n}{3}} \right), \quad \text{as}\quad n \to \infty.
\end{align}
This technique was later refined by Rademacher \cite{HRad1937, HRad1943, HRad1973} and quite remarkably he obtained an exact absolutely convergent series for $p(n)$, 
\begin{align}\label{exact formula for p(n)}
p(n) = 2\pi \left(\frac{1}{6\sqrt{\frac{2}{3}\left(n-\frac{1}{24}\right)}}\right)^{\frac{3}{2}} \sum_{k=1}^{\infty} \frac{A_{k}(n)}{k} I_{\frac{3}{2}}\left(\frac{\pi}{k}\sqrt{\frac{2}{3}\left(n-\frac{1}{24}\right)} \right),
\end{align}
where $I_{\nu}$ denotes the modified Bessel function of the first kind and $A_k(n)$ is a Kloosterman-type sum involving exponential terms,  defined as
\begin{align*}
A_{k}(n) = \sum_{\substack{h ~\mathrm{mod}~k \\ (h,k)=1}} e^{i\pi s(h,k) -2\pi i n \frac{h}{k}},
\end{align*}
with $s(h,k)$ being the Dedekind sum given by 
\begin{align}\label{dedekind_sum}
s(h,k):= \sum_{r=1}^{k-1}\frac{r}{k} \left( \frac{hr}{k} - \left\lfloor \frac{hr}{k} \right\rfloor - \frac{1}{2} \right).
\end{align}
It is worth noting that Rademacher  \cite[Equation (1.8)]{HRad1937} first used the path considered by Hardy and Ramanujan to derive the above exact formula \eqref{exact formula for p(n)} for $p(n)$.   Later,  he \cite{HRad1943} modified the path using Ford circles based on Farey fractions in the unit interval,  which we will briefly discuss in Section \ref{Prilim}. 
The first term of the series \eqref{exact formula for p(n)} aligns with the asymptotic formula \eqref{asymptotic formula for p(n)} established by Hardy and Ramanujan.  Their proof fundamentally relied on the fact that the generating function for the  partition function has a connection with the Dedekind eta function,  which is a half integral weight modular form and its transformation formula played a crucial role.   Along with Zuckerman \cite{RZ38},  Rademacher extended his method for the Fourier coefficients of certain modular forms of positive weight.   
Building upon this approach,  Zuckerman \cite{Zucker39}  derived exact convergent series for the Fourier coefficients, at any cusp, of all weakly holomorphic modular forms of negative weight associated with any congruence subgroup of $SL_2(\mathbb{Z})$.  In 2012,  Bringmann and Ono \cite{BO2011} further extended the work of Rademacher and Zuckerman for the coefficients of harmonic Maass forms of weight less than or equal to half.  

\subsection{Rademacher-type exact formula}

Hardy-Ramanujan-type asymptotic formula and Rademacher-type exact formula for various partition functions have been investigated   by several mathematicians including Andrews \cite{Andrews},  Bringmann-Ono \cite{BO06},  Bringmann-Mahlburg \cite{BM11},   Grosswald \cite{Gros58, Gros60},  Hagis \cite{Hagis62}-\cite{Hagis71}, Hua \cite{LKHUA},  Iseki \cite{I61},  Niven \cite{N40},  and Sills \cite{Sills2010}.  For the modern framework of Rademacher-type expansions for the coefficients of harmonic Maass forms, we refer the reader to \cite{BFOR17}.

Over time, numerous analogues and generalizations of the partition function have been introduced. One such example is the \emph{cubic partition function} introduced by Chan \cite{Chan10-1} in relation to Ramanujan's cubic continued fraction.   He further studied Ramanujan-type congruences  \cite{Chan10-2, Chan11} modulo prime powers. 
This function is denoted by $a(n)$, which counts the number of ways $n$ can be written as sum of natural numbers where even numbers can appear in two different colors. For example, the number $3$ has $4$ such partitions, namely, $3,~2_R+1,~2_B+1$ and $1+1+1$, where the subscripts $R$ and $B$ denote the colors Red and Blue, respectively. The generating function for $a(n)$ is given by
\begin{align*}
\sum_{n=0}^\infty a(n) q^n = \frac{1}{(q;q)_\infty (q^2;q^2)_\infty}.
\end{align*}
The arithmetic properties of the cubic partition function, particularly Ramanujan-style congruences for various primes have been investigated extensively; for further details, we refer the reader to \cite{ASS25,  CD17, Hir20, Lin17, Mer25, Sel14} and the references therein.
In 2024,    Mauth \cite[Theorem 3.1]{Mauth} applied Zuckerman's technique to derive an exact formula for $a(n)$.

In recent years, the study of overpartitions, first introduced by Corteel and Lovejoy \cite{overpartition}, has become an active area of research in partition theory.
Overpartitions are partitions in which the first occurrence (or equivalently last occurrence) of each distinct part may be overlined. This concept has yielded numerous combinatorial and analytic applications, especially in the theory of mock modular forms, ranks, and cranks.  A Rademacher-type exact formula for overpartitions  had already been established by Zuckerman  \cite[p.~321, Equation (8.53)]{Zucker39} way back in 1939,  which was later explicitly given by Sills \cite{Sills2010},  
\begin{align*}
\overline{p}(n) = \frac{\pi}{4\sqrt{2} n^{\frac{3}{4}}} \sum_{\substack{k=1 \\ k~\mathrm{odd}}}^\infty \frac{B_{k}(n)}{k} I_{\frac{3}{2}} \left( \frac{\pi \sqrt{n}}{k} \right),  
\end{align*}
where
\begin{align*}
B_k(n)=  \sum_{\substack{h ~\mathrm{mod}~k \\ (h,k)=1}} e^{\pi i( 2 s(h,k)- s(2h,k)) -2\pi i n \frac{h}{k}},
\end{align*}
and $s(h,k)$ is the Dedekind sum defined in \eqref{dedekind_sum}.

An overpartition analogue of the cubic partition function was introduced by Kim \cite{Kim}, where the first occurrence of a part in a cubic partition may be overlined.  If parts are repeated, only one of them is allowed to be over lined. For example, there are $12$ cubic overpartitions of $3$, namely, $3,~\bar{3},~2_R+1,~\bar{2}_R+1,~2_R+\bar{1},~\bar{2}_R+\bar{1},~2_B+1,~ \bar{2}_B+1,~2_B+\bar{1},  ~\bar{2}_B+\bar{1},  ~1+1+1,$ and $\bar{1}+1+1$. The function counting such partitions is denoted by $\overline{a}(n)$ and its generating function is given by
\begin{align}\label{generating function of cubic overpartition}
\sum_{n=0}^\infty \overline{a}(n) q^n = \frac{(-q;q)_\infty (-q^2;q^2)_\infty}{(q;q)_\infty (q^2;q^2)_\infty}.
\end{align}
Despite the significant progress in exact formulas for various partition functions, a Rademacher-type exact formula for the cubic overpartition function $\overline{a}(n)$ has remained elusive. In this paper, we rigorously establish an exact formula for $\overline{a}(n)$ expressed as a sum of two convergent series of Bessel functions weighted by  Kloosterman-type sums. This result not only extends Rademacher's classical work but also illustrates the modular richness inherent in cubic overpartitions.

\subsection{Higher order Tur\'{a}n  inequalities}

In addition to exact formulas and asymptotics, another important direction in the study of partition functions is the investigation of Tur\'{a}n inequalities. A sequence $\{t(n)\}$ of real numbers is said to be log-concave if it satisfies the Tur\'{a}n inequality:
\begin{align*}
t(n)^2 \geq t(n-1)t(n+1),\quad \text{for}~ \text{all} \quad n \geq 1. 
\end{align*}
The study of log-concavity and Tur\'{a}n inequalities play a pivotal role in combinatorics, number theory, and are deeply connected to the theory of real entire functions in the Laguerre–P\'{o}lya class, as well as to aspects of the Riemann Hypothesis as discussed in \cite{CNV86, Dimitrov,  GORZ19, szego}.  Beyond log-concavity,  the study of higher-order Turán inequalities finds a natural framework within the theory of Jensen polynomials.

For a real sequence $\{t(n)\}$, the Jensen polynomial of degree $d$ and shift $n$ is defined as
\begin{align*}
J_t^{d,n}(X) := \sum_{i=0}^{d} \binom{d}{i} t(n+i)\, X^i.
\end{align*}
We say that $t(n)$ satisfies the degree $d$ Tur\'{a}n inequality at $n$ if the Jensen polynomial $J_t^{d,n-1}(X)$ is hyperbolic i.e.,  has only real roots.

Mathematicians have been interested in studying Tur\'{a}n inequalities for various partition functions due to their deep connections with log-concavity, real-rootedness of polynomials, and asymptotic behaviour of combinatorial sequences.  Log-concavity is a common property observed in many sequences arising in combinatorics.  Prominent examples include the binomial coefficients, the Stirling numbers, and the Bessel numbers,  see \cite{S89}.   For the ordinary partition function $p(n)$, the following inequality:
 $$ 
 p(n)^2 > p(n-1)p(n+1)
 $$ 
 was first established by Nicolas \cite{Nicolas78} for all $n>25$.  This result was later reproved by DeSalvo and Pak \cite{DeSalvo2015} using Lehmer's estimate on the error term of the Hardy-Ramanujan-Rademacher formula for $p(n)$.  

Recently,  Chen, Jia and Wang \cite{CJW2019} established the hyperbolicity of $J_{p}^{3,n-1}(X)$ for all  $n \geq 94$.  Further,  in the same paper,   they conjectured that for any integer $d \geq 1$ there exists an integer $N_p(d)$ such that $J_{p}^{d,n-1}(X)$ is hyperbolic for $n \geq N_p(d)$. 

This conjecture was settled by Griffin, Ono, Rolen and Zagier \cite[Theorem 5]{GORZ19}.  Moreover,  they proved the hyperbolicity of Jensen polynomials for a wider class of sequences, namely for the Fourier coefficients of weakly holomorphic modular forms over the full modular group $SL_2(\mathbb{Z})$.  They further connected their result via a more general phenomenon, in which Jensen polynomials  for a class of sequences can be modelled by the Hermite polynomials $H_d(X)$,  which is defined as
\begin{align*}
\sum_{d=0}^{\infty} H_d(X) \frac{t^d}{d!} = e^{-t^2 + Xt} = 1 + X t + \frac{X^2 - 2}{2!} t^2 + \frac{X^3 - 6X}{3!} t^3 + \cdots.
\end{align*}
They proved the following beautiful result.  
\begin{theorem}\label{theorem of Griffin ono rolen and zagier} \cite[Theorem 3 and 8]{GORZ19}
Let $\{\alpha(n)\}, \{A(n)\}, \{\delta(n)\}$ be sequences of positive real numbers, with $\delta(n)$ tending to $0$.  
For integers $j \geq 0$,  $d \geq 3$, suppose that there are real numbers $g_3(n), g_4(n), \ldots, g_d(n)$, for which 
\begin{align*}
\log \left( \frac{\alpha(n+j)}{\alpha(n)} \right)
= A(n)j - \delta(n)^2 j^2 + \sum_{i=3}^{d} g_i(n) j^i + o(\delta(n)^d),
\quad \text{as } n \to \infty,
\end{align*}
with $g_i(n) = o(\delta(n)^i)$ for each $3 \leq i \leq d$.  
Then we have  
\begin{align*}
\lim_{n \to \infty} \left( \frac{\delta(n)^{-d}}{\alpha(n)}
J_{\alpha}^{d,n} \!\left( \frac{\delta(n)X - 1}{\exp(A(n))} \right) \right)
= H_d(X).
\end{align*}
\end{theorem}
It is well-known that 
the Hermite polynomials have distinct real roots and this property of a polynomial
with real coefficients is invariant under linear any transformation,  which implies that Jensen polynomials are hyperbolic for large values $n$.

Log-concavity for overpartitions was established by Engel \cite{Engel2017} in 2017.   Subsequently, Liu and Zhang \cite{LiuZhang2021} studied higher-order Turán inequalities for overpartitions. More recently, log-concavity, third-order Turán inequalities, and strict log-subadditivity for cubic partitions were examined by Li, Peng, and Zhang \cite{cubicturan}.  In 2024,  Dong and Ji \cite{DJ2024} established log-concavity and the third-order Turán inequality for the distinct partition function. Higher-order Turán inequalities for MacMahon's plane partitions were investigated by Ono, Pujahari, and Rolen \cite{OPR2022}.

The Hardy–Ramanujan-type asymptotic formula does not directly yield the log-concavity property for the corresponding partition function. In contrast, the Rademacher-type exact formula provides refined error terms that can be employed to establish log-concavity. Motivated by this, we first establish a Rademacher-type exact formula and, as an application, prove the second-order Turán inequality for cubic overpartitions. Furthermore, by applying the above result of Griffin, Ono, Rolen, and Zagier, we establish higher-order Turán inequalities for the cubic overpartition function $\overline{a}(n)$.

The structure of the paper is as follows. In Section~\ref{main results}, we present the main findings of this paper.  Section~\ref{Prilim} compiles the preliminary results required for the proofs; in particular, we discuss properties of Farey sequences and Ford circles,  and we derive transformation formulas for $f(q^2)$ and $f(q^4)$, where 
$q = e^{ 2\pi i \left(\tfrac{h}{k} + \tfrac{iz}{k^2} \right)}, \ \tfrac{h}{k} \in \mathbb{Q}, \ \Re(z) > 0$. 
Section~\ref{proof of main result} is devoted to establishing the results stated in Section~\ref{main results}. In Section~\ref{verification table}, we provide numerical verification for Theorem~\ref{exact formula for overcubic partition theorem}. Finally, in the concluding remarks, we present a few observations and conjectures that may be of independent interest to the reader.

\section{Main results}\label{main results}
We begin this section with an exact formula for the cubic overpartition function $\overline{a}(n)$, which mirrors the classical Rademacher series for the partition function.
\begin{theorem}\label{exact formula for overcubic partition theorem}
Let $\overline{a}(n)$ be the number of cubic overpartitions of a non-negative integer $n$. Then, $\overline{a}(n)$ admits the following exact formula: 
\begin{align}\label{exact formula for overcubic partition equation}
\overline{a}(n) &= \frac{3\pi}{16n \sqrt{2}} \sum_{\substack{k=1 \\ k~\text{odd}}}^\infty \frac{A_k^{(1)}(n)}{k} I_2\left( \frac{\pi}{k} \sqrt{\frac{3n}{2}} \right) + \frac{\pi}{4n \sqrt{2}} \sum_{\substack{k=1 \\ k\equiv 2 ~(\textup{mod}~4)}} ^\infty \frac{A_k^{(2)}(n)}{k} I_2\left( \frac{\pi}{k} \sqrt{2n} \right),
\end{align}
where
\begin{align*}
A_k^{(1)}(n)&=\sum_{\substack{h=0 \\ (h,k)=1}}^{k-1}  e^{ \pi i \left(2s(h,k) + s(2h,k)-s(4h,k) \right) -2n \pi i \frac{h}{k} },\\
A_k^{(2)}(n)&=\sum_{\substack{h=0 \\ (h,k)=1}}^{k-1}  e^{ \pi i \left(2s(h,k) + s\left(h,\frac{k}{2}\right)-s\left(2h,\frac{k}{2}\right) \right)-2n \pi i \frac{h}{k}},
\end{align*}
and $I_{\nu}$ is the modified Bessel function of first kind and $s(h,k)$ is  the Dedekind sum defined in \eqref{dedekind_sum}.
\end{theorem}
\begin{remark}
 The formula \eqref{exact formula for overcubic partition equation} is structurally analogous to Rademacher’s exact formula for the ordinary partition function $p(n)$, with the primary difference being the presence of two distinct series corresponding to $k$ odd,  and $k\equiv 2   \pmod 4 $.  
\end{remark}
As a direct consequence of Theorem \ref{exact formula for overcubic partition theorem}, we  deduce an asymptotic formula of $\overline{a}(n)$ for large $n$ by analyzing the dominant contribution from the leading terms of the two series. This leads to the following asymptotic formula.
\begin{corollary}\label{corollary of asymptotic of cubic overpartition}
The cubic overpartition function $\overline{a}(n)$ 
satisfies the following estimate: 
\begin{align}\label{asymptotic for cubic}
\overline{a}(n) =\frac{3\pi}{16n\sqrt{2}} I_2\left(\pi \sqrt{\frac{3n}{2}} \right) + \mathcal{O}\left(\bigg( \frac{2\cdot 5^{\frac{5}{2}}}{\pi n^{\frac{7}{4}}} + \frac{1}{n^{\frac{5}{4}}} \bigg) e^{\pi \sqrt{\frac{n}{2}}} \right).
\end{align}
Further,  we have 
\begin{align}\label{asymptotic cubic overpartition}
\overline{a}(n) \sim \frac{3^{\frac{3}{4}}}{2^{\frac{19}{4}}n^{\frac{5}{4}}} e^{\pi \sqrt{\frac{3n}{2}}}, \quad \text{as} \quad  n \to \infty.
\end{align}
\end{corollary}

Before presenting the next result, we introduce a few variables which will be used frequently.  We define
\begin{align}\label{Definition of v}
v := \pi \sqrt{\frac{3n}{2}}, \quad
v^+ := \pi \sqrt{\frac{3(n+1)}{2}}, \quad
v^- := \pi \sqrt{\frac{3(n-1)}{2}}.
\end{align}
In our pursuit of establishing Tur\'{a}n inequality for cubic overpartitions, the following result plays an important role. 
\begin{theorem}\label{product of cubic overpartition bound}
For all $n \geq 2363$, we have 
\begin{align*}
\Upsilon_1(n) \leq \frac{\overline{a}(n+1)\overline{a}(n-1)}{(\overline{a}(n))^2} \leq \Upsilon_2(n),
\end{align*}
where
\begin{align*}
\Upsilon_1(n):= 1- \frac{9\pi^4}{16 v^3} + \frac{45 \pi^4}{16 v^4} -\frac{309}{v^5} - \frac{535}{v^6} - \frac{405 \pi^8}{2048 v^6} - \frac{729 \pi^{12}}{ v^6},
\end{align*}
\begin{align*}
\Upsilon_2(n):= 1- \frac{9\pi^4}{16 v^3} + \frac{45 \pi^4}{16 v^4} -\frac{308}{v^5} - \frac{286}{v^6} + \frac{81 \pi^8}{256 v^6} + \frac{729 \pi^{12}}{16 v^6}.  
\end{align*}

The next result gives the second order Tur\'{a}n inequality for $\overline{a}(n)$.  

\end{theorem}
\begin{theorem}\label{log concavity of cubic overpartitions}
For $n \geq 10$,  we have $\overline{a}^2(n) >  \overline{a}(n+1) \overline{a}(n-1)$. 
\end{theorem}

More generally,  we prove the following result,  which in turn gives higher Tur\'{a}n inequalities for cubic overpartitions.

\begin{theorem}\label{hyperbolicity of Jensen polynomial}
For any positive integer $d \geq 3$, $J_{\overline{a}}^{d,n}(X)$ is hyperbolic for all but finitely many values of n.
\end{theorem}

Bessenrodt and Ono \cite{BO16},  in 2016,  proved the following log-subadditivity result for the ordinary partition function i.e.,   for any integers $n,  m >1$ with $n+m>8$,  the inequality 
\begin{align*}
p(n) p(m) \geq p(n+m),  
\end{align*}
is true and the equality holds only when $\{n,  m\}=\{2,7\}$.  Inspired from this result,  we state the following log-subadditivity result for the cubic overpartition function. 
\begin{theorem}\label{subadditivity for overcubic partitions}
 For any $n,m \geq 1$ and  $\{n,  m\} \neq \{1,1\},  \{1,3\}$, we have
 \begin{align*}
 \overline{a}(n) \overline{a}(m) \geq \overline{a}(n+m).  
 \end{align*}
Moreover,  the equality holds only when $\{n,  m\}=\{1,2\}$.  
\end{theorem}

In 2010,  Chen \cite{ChenTalk2010} conjectured that the partition function $p(n)$ satisfies the generalized log-concavity property, namely, for integers $n>m>1$,
\begin{align*}
p(n)^2 >  p(n-m)p(n+m).
\end{align*}
This conjecture was later proved by DeSalvo and Pak \cite{DeSalvo2015}. Motivated by this result,  we prove an analogous result for the cubic overpartition function $\overline{a}(n)$ as follows:
\begin{theorem}\label{general log concavity}
For all $n>m > 1$,  we have
\begin{align*}
\overline{a}(n)^2 >  \overline{a}(n-m)\overline{a}(n+m).
\end{align*}
\end{theorem}

\section{Preliminaries}\label{Prilim}
In order to apply the circle method to derive an exact formula for the cubic overpartition function $\overline{a}(n)$, it is essential to carefully choose an appropriate path of integration in the complex plane. This path is typically constructed using Farey fractions and Ford circles. In this section, we briefly recall these classical concepts from the work of Hardy-Ramanujan \cite{HR1918} and Rademacher \cite{HRad1973}.

\subsection{Farey sequences and Ford circles} \label{Ford circles}
A Farey sequence of order $n$, denoted as $F_n$, is an ascending sequence of reduced fractions in $[0,1]$, whose denominator do not exceed $n$. Formally,  one can write
\begin{align*}
F_n = \left\{ \frac{h}{k} \in \mathbb{Q} \cap [0,1] ~\middle|~ 0 \leq h \leq k \leq n,~ \gcd(h,k)=1 \right\}.
\end{align*} 
We can associate a geometrical object to each Farey fraction $\frac{h}{k}$,  known as the Ford circle, denoted by $C(h,k)$,  in the complex plane with centre at $\left(\frac{h}{k}, \frac{1}{2k^2}\right)$ and radius $\frac{1}{2k^2}$. That is,  
\begin{align}\label{Ford circle C(h,k)}
C(h,k) := \left\{ z \in \mathbb{C} ~\Big|~ \left| z - \left( \frac{h}{k} + \frac{i}{2k^2} \right) \right| = \frac{1}{2k^2} \right\}.  
\end{align}
Several properties of Ford circles are noteworthy to mention.    The ford circle $C(h,k)$ lies entirely in the upper half plane and the real axis  acts as a tangent at $( \frac{h}{k}, 0)$.  No two Ford circles intersect.   One can show that if $\frac{h_1}{k_1} < \frac{h}{k} < \frac{h_2}{k_2}$ are three adjacent fractions in a Farey sequence, then $C(h,k)$ will be tangent to $C(h_1,k_1)$ and $C(h_2,k_2)$. Let $\alpha_1$ and $\alpha_2$ be the points of tangency.
We can easily show that 
\begin{align*}
\alpha_1= \frac{h}{k} - \frac{k_1}{k(k^2 + k_1^2)} + i \frac{1}{k^2 + k_1^2},  \quad \alpha_2 =  \frac{h}{k} + \frac{k_2}{k(k^2 + k_2^2)} + i \frac{1}{k^2 + k_2^2}.  
\end{align*}
The Farey sequence provides a systematic way to dissect the unit interval into subintervals corresponding to rational cusps, while Ford circles are non-overlapping regions associated to each cusp. While applying the Hardy-Ramanujan-Rademacher circle method for the generating function of cubic overpartitions,  we choose Rademacher's path of
 contour integration in which the path is carefully chosen to follow arcs of Ford circles up to a height determined by the Farey sequence of order $N$. This decomposition of the path of integration allows to find precise approximations of the generating function near each cusp, which is the cornerstone of deriving an exact formula for $\overline{a}(n)$. 

We now state an important result about Dedekind sums, which will be a key ingredient in our derivation of the transformation formula.
\begin{lemma}\label{important_result_for_dedekind_sum}
If $h_1$ is an integer such that $hh_1 \equiv 1 (\textup{mod}~ k)$, then $s(h,k) = s(h_1,k) $.
\end{lemma}
\begin{proof}
See \cite[Equation (68.5)]{IK2004} for a proof this result.
\end{proof}

Next, we establish transformation formula for $f(q^2)$ and $f(q^4)$, with $q= e^{ 2\pi i \left(\frac{h}{k} + \frac{iz}{k^2} \right)}$,  $\frac{h}{k} \in \mathbb{Q},  \Re(z)>0$,  which will play crucial role in obtaining the exact formula for cubic overpartitions.   

It is well-known that the Dedekind eta function $n(\tau)$ has a close connection with the partition generating function $f(q)$.  Mainly,   for $\tau \in \mathbb{H}$,  it is defined by
\begin{align}\label{relation between f and eta}
\eta (\tau) = e^{ \frac{ i \pi \tau}{12}}  \prod_{m=1}^{\infty}(1-e^{2\pi i m \tau}) = \frac{e^{ \frac{i \pi \tau}{12}}}{f(e^{2\pi i \tau})}.  
\end{align}
To find a transformatian formula for $f(q^2)$,  we replace $\tau$ by $\frac{2}{\gamma} \left( \frac{h}{k}+ \frac{iz}{k}\right)$ in \eqref{relation between f and eta} for any $\gamma \in \{1,2\}$.  Thus,  we have
\begin{align}\label{relation of f and eta for f(q^2)}
f\left(e^{\frac{4\pi i}{\gamma} \left(\frac{h}{k} + \frac{iz}{k} \right)}\right) = e^{\frac{i \pi }{6 \gamma} \left(\frac{h}{k} + \frac{iz}{k} \right) } \eta^{-1} \left( \frac{2}{\gamma}  \left(\frac{h}{k} + \frac{iz}{k} \right) \right).
\end{align}
Now we state a transformation formula for $\eta(\tau)$,  namely,  for  $\tau' \in \mathbb{H}$,  
\begin{align}\label{transformation_formula_of_eta}
\eta\left(\frac{a\tau' +b}{c \tau' +d}\right) = \epsilon(a,b,c,d) \sqrt{\frac{c\tau' +d}{i}} \eta(\tau'), \quad \forall \begin{bmatrix}
a & b \\ c & d 
\end{bmatrix} \in SL_2(\mathbb{Z}),  
\end{align}
where $\epsilon (a,b,c,d)$ is defined as 
\begin{align*}
\epsilon (a,b,c,d) = \exp \Bigg( \frac{i \pi }{12}\Phi \begin{bmatrix}
a & b \\ c & d 
\end{bmatrix} \Bigg),
\end{align*}
with
\begin{align*}
\Phi \begin{bmatrix}
a & b \\ c & d 
\end{bmatrix} = \begin{cases} 
b + 3, ~~& \mathrm{for} ~~ c=0,  d=1, \\
-b - 3, ~~& \mathrm{for} ~~ c=0,  d=-1,   \\
 \frac{a+d}{c}-12~~ \mathrm{sign}(c)  s(d,|c|),~~ & \mathrm{for} ~~ c\neq 0,
\end{cases}
\end{align*} 
and $s(h,k)$ is the Dedekind sum defined in \eqref{dedekind_sum}.  
To utilize the transformation formula \eqref{transformation_formula_of_eta} for $\eta(\tau)$, we need to have
\begin{align*}
\frac{a\tau' +b}{c \tau' +d}=\frac{2}{\gamma} \left(\frac{h}{k} + \frac{iz}{k} \right),
\end{align*}
which suggest that,  one should consider $\tau' = \frac{h'}{k} + \frac{i}{2 kz}$ for some $h' \in \mathbb{Z}$ and 
\begin{align}\label{values of abcd for f(q^2)}
\begin{bmatrix}
a & b \\ c & d
\end{bmatrix} = \begin{bmatrix}
\frac{2}{\gamma}h & -\frac{\frac{2}{\gamma} hh'+1}{k} \\ k & -h'
\end{bmatrix}.  
\end{align}
In order for $\frac{\frac{2}{\gamma} hh'+1}{k}$ to be an integer, we impose the condition that $\frac{2}{\gamma} hh' \equiv -1(\mathrm{mod}~k)$. Note that $c=k\neq0$ implies that $\mathrm{sign}(k) = 1$. Therefore,  we have
\begin{align}\label{epsilon (a,b,c,d) for f(q^2)}
\epsilon (a,b,c,d)=\exp \left(\frac{i\pi}{12}\left(\frac{\frac{2}{\gamma} h-h'}{k}-12s(-h',k)\right) \right).
\end{align}
Now we make use of Lemma \ref{important_result_for_dedekind_sum} to see that $s(-h',k) = s\left(\frac{2 h}{\gamma}, k\right)$. Hence,  utilizing \eqref{transformation_formula_of_eta}, \eqref{values of abcd for f(q^2)} and \eqref{epsilon (a,b,c,d) for f(q^2)} in \eqref{relation of f and eta for f(q^2)},  we get
\begin{align*}
f\left(e^{\frac{4\pi i}{\gamma} \left(\frac{h}{k} + \frac{iz}{k} \right)}\right) = \sqrt{\frac{2 z}{\gamma}} e^{i\pi s\left(\frac{2 h}{\gamma},  k\right) + \frac{\pi}{12k}\left(\frac{\gamma}{2 z}-\frac{2 z}{\gamma}\right)}f\left(e^{2\pi i \left(\frac{h'}{k}+ \frac{i\gamma}{2 kz}\right)}\right).
\end{align*}
Now,  replace $k$ by $\frac{k}{\gamma}$ and then $z$ by $\frac{z}{k}$ to have
\begin{align*}
f\left(e^{4\pi i \left(\frac{h}{k} + \frac{iz}{k^2} \right)}\right) = \sqrt{\frac{2 z}{k \gamma}} e^{i\pi s\left(\frac{2 h}{\gamma},\frac{k}{\gamma}\right) + \frac{\pi}{12k}\left(\frac{\gamma^2k}{2 z}-\frac{2 z}{k}\right)}f\left(e^{2\pi i \gamma\left(\frac{h'}{k}+ \frac{i\gamma}{2 z}\right)}\right),  
\end{align*}
where $\frac{2}{\gamma} hh' \equiv -1(\mathrm{mod}~\frac{k}{\gamma})$. As $\gamma \in \{1,2\}$, the transformation formula for $f(q^2)$ becomes
\begin{align}\label{transformation formula for f(q^2)}
f\left(e^{4 \pi i \left(\frac{h}{k} + \frac{iz}{k^2}\right)}\right) = 
\begin{cases}
\sqrt{\frac{2z}{k}} e^{i\pi s(2h,k)} e^{ \frac{\pi}{12k} \left(\frac{k}{2z} - \frac{2z}{k} \right)}f\left(e^{2\pi i \left(\frac{h_2}{k}+ \frac{i}{2z}\right)}\right),~~&\mathrm{if}~~ k~\mathrm{odd},\\
\sqrt{\frac{z}{k}} e^{i\pi s(h,\frac{k}{2})} e^{\frac{\pi}{6k}\left(\frac{k}{z}-\frac{z}{k}\right)}f\left(e^{4\pi i \left(\frac{h_3}{k}+ \frac{i}{z}\right)}\right),~~&\mathrm{if}~~k~\mathrm{even},
\end{cases}
\end{align}
where $2hh_2 \equiv -1$ (mod $k$) and $hh_3 \equiv -1$ (mod $\frac{k}{2}$).

In a similar way we find the transformation formula for $f(q^4)$.
For any $\delta \in \{1,  2,  4 \}$,   we replace $\tau$ by $\frac{4}{\delta} \left( \frac{h}{k}+ \frac{iz}{k}\right)$ in \eqref{relation between f and eta} to see that 
\begin{align}\label{relation of f and eta for blr(n)}
f\left(e^{\frac{8\pi i }{\delta} \left(\frac{h}{k} + \frac{iz}{k} \right)}\right) = e^{\frac{i \pi }{3 \delta} \left(\frac{h}{k} + \frac{iz}{k} \right) } \eta^{-1} \left( \frac{4}{\delta}  \left(\frac{h}{k} + \frac{iz}{k} \right) \right).
\end{align}
To utilize \eqref{transformation_formula_of_eta}, we write
\begin{align*}
\frac{a\tau' +b}{c \tau' +d}=\frac{4}{\delta} \left(\frac{h}{k} + \frac{iz}{k} \right),
\end{align*}
which indicates that,   $\tau' = \frac{h''}{k} + \frac{i}{4 kz}$ for some $h'' \in \mathbb{Z}$ and 
\begin{align}\label{values of abcd for blr(n)}
\begin{bmatrix}
a & b \\ c & d
\end{bmatrix} = \begin{bmatrix}
\frac{4}{\delta}h & -\frac{\frac{4}{\delta} hh''+1}{k} \\ k & -h''
\end{bmatrix}.  
\end{align}
 We force the condition $\frac{4}{\delta} hh'' \equiv -1(\mathrm{mod}~k)$,  so that all entries are integers. 
 Note that $\mathrm{sign}(k) = 1$.  Hence,   we have
\begin{align}\label{epsilon (a,b,c,d) for blr(n)}
\epsilon (a,b,c,d)=\exp \left(\frac{i\pi}{12}\left(\frac{\frac{4}{\delta} h-h''}{k}-12s(-h'',k)\right) \right).
\end{align}
Here we utilized Lemma \ref{important_result_for_dedekind_sum} to have $s(-h'',k) = s\left(\frac{4 h}{\delta}, k\right)$.  Now employing \eqref{transformation_formula_of_eta}, \eqref{values of abcd for blr(n)} and \eqref{epsilon (a,b,c,d) for blr(n)} in \eqref{relation of f and eta for blr(n)},  we arrive at
\begin{align*}
f\left(e^{\frac{8 \pi i }{\delta} \left(\frac{h}{k} + \frac{iz}{k} \right)}\right) = \sqrt{\frac{4 z}{\delta}} e^{i\pi s\left(\frac{4 h}{\delta},  k\right) + \frac{\pi}{12k}\left(\frac{\delta}{4 z}-\frac{4 z}{\delta}\right)}f\left(e^{2\pi i \left(\frac{h''}{k}+ \frac{i\delta}{4 kz}\right)}\right).
\end{align*}
Substituting $k$ by $\frac{k}{\delta}$ and then $z$ by $\frac{z}{k}$,  we get
\begin{align*}
f\left(e^{8\pi i  \left(\frac{h}{k} + \frac{iz}{k^2} \right)}\right) = \sqrt{\frac{4 z}{k \delta}} e^{i\pi s\left(\frac{4 h}{\delta},\frac{k}{\delta}\right) + \frac{\pi}{12k}\left(\frac{\delta^2k}{4 z}-\frac{4 z}{k}\right)}f\left(e^{2\pi i \delta\left(\frac{h''}{k}+ \frac{i\delta}{4 z}\right)}\right),  
\end{align*}
where $\frac{4}{\delta} hh'' \equiv -1(\mathrm{mod}~\frac{k}{\delta})$.   
Here,  we must pay attention to the fact that for each $\delta \in \{1, 2,4\}$ with $\gcd(k, 4) = \delta$,  we will get a different transformation formula.   This gives  the final transformation formula for $f(q^4)$ as follows: 
\begin{align}\label{transformation formula for f(q^4)}
f\left(e^{8 \pi i \left(\frac{h}{k} + \frac{iz}{k^2}\right)}\right) = 
\begin{cases}
\sqrt{\frac{4 z}{k }} e^{i\pi s\left(4h,k\right)} e^{\frac{\pi}{12k}\left(\frac{k}{4z}-\frac{4 z}{k}\right)}f\left(e^{2\pi i \left(\frac{h_4}{k}+ \frac{i}{4 z}\right)}\right),~~&\mathrm{if}~~ k~\mathrm{odd},\\
\sqrt{\frac{2z}{k}} e^{i\pi s\left(2h,\frac{k}{2}\right)} e^{\frac{\pi}{12k}\left(\frac{k}{z}-\frac{4z}{k}\right)}f\left(e^{4\pi i \left(\frac{h_5}{k}+ \frac{i}{2z}\right)}\right),~&\mathrm{if}~ k \equiv 2 (\textup{mod}~4),\\
\sqrt{\frac{z}{k}} e^{i\pi s\left(h,\frac{k}{4}\right)} e^{\frac{\pi}{3k}\left(\frac{k}{z}-\frac{z}{k}\right)}f\left(e^{8\pi i \left(\frac{h_6}{k}+ \frac{i}{z}\right)}\right),~~&\mathrm{if}~~ k \equiv 0 (\mathrm{mod}~4),
\end{cases}
\end{align}
where $4hh_4 \equiv -1$ (mod $k$), $2hh_5 \equiv -1$ (mod $\frac{k}{2}$), and $hh_6 \equiv -1$ (mod $\frac{k}{4}$).

The next result gives an upper and lower bound for $I_2(s)$. 
\begin{lemma}
Define the expression $E_{I_2}(s)$ as follows:
\begin{align}\label{Def_EI2}
E_{I_2}(s) := 1 - \frac{15}{8s}+ \frac{105}{128s^2} + \frac{315}{1024s^3} + \frac{10395}{32768s^4} + \frac{135135}{262144s^5}, 
\end{align}
then for $ s \geq 25 $, the bounds for $I_2(s)$ are given by
\begin{align}\label{bound_I2}
\frac{e^s}{\sqrt{2\pi s}} \left(E_{I_2}(s)-\frac{31}{s^6}\right) \leq I_2(s) \leq \frac{e^s}{\sqrt{2\pi s}} \left(E_{I_2}(s)+\frac{31}{s^6}\right).            
\end{align}
\end{lemma}
\begin{proof}
For a proof of this result we refer \cite[Lemma 3.1]{cubicturan}.   
\end{proof}
We now present a key result that provides an upper bound for the tail of a series involving modified Bessel functions of the first kind. This bound plays a crucial role to obtain bounds for the error term of cubic overpartitions.
\begin{lemma}
Let $N$ be a positive integer. Then, we have
\begin{align}\label{bound_for_I2_series}
\sum_{j=N+1}^{\infty}I_2\left(\frac{s}{j} \right) \leq \frac{2N^2}{s} I_1\left(\frac{s}{N} \right).
\end{align}
\end{lemma}
\begin{proof}
For details of the proof, one can see \cite[Equation (2.2)]{cubicturan}.
\end{proof}
The next result presents a lower bound for $I_2(s)$.
\begin{lemma}
For all $s \geq 30$, we have
\begin{align}\label{lower bound for I2}
I_2(s) \geq \frac{e^s}{\sqrt{2\pi s}} \left(1+\frac{2}{s} \right)^{-1}.
\end{align}
\end{lemma}
\begin{proof}
We refer \cite[Equation (3.14)]{cubicturan} to get insights of the proof.
\end{proof}
Having established the necessary background on Farey fractions, Ford circles,  Dedekind sums,  and properties of Bessel functions,  we are now prepared to proceed towards the proof of the main results of this paper. In the next section, we will apply these results, in conjunction with the circle method, to rigorously prove the Rademacher-type exact formula and  Tur\'{a}n inequalities  for the cubic overpartition function.

\section{Proof of main results}\label{proof of main result}

\subsection{Rademacher-type exact formula for cubic overpartitions}
In this subsection, we present a detailed proof of Rademacher-type exact formula for cubic overpartitions by applying the circle method, leveraging the transformation formula derived earlier, and carefully analyzing contributions from different Farey arcs associated with  Ford circles.

\begin{proof}[Theorem \ref{exact formula for overcubic partition theorem}][]
Let $\bar{A}(q)$ denotes the generating function for $\overline{a}(n)$. Then from \eqref{generating function of cubic overpartition},   we have
\begin{align*}
\overline{A}(q):= \sum_{n=0}^\infty \overline{a}(n) q^n = \frac{(-q;q)_\infty (-q^2;q^2)_\infty}{(q;q)_\infty (q^2;q^2)_\infty} = \frac{(q^4;q^4)_\infty}{(q;q)^2_\infty(q^2;q^2)_\infty} = \frac{(f(q))^2 f(q^2)}{f(q^4)},
\end{align*}
where $f(q)= \frac{1}{(q;q)_\infty}$. Now applying Cauchy's integral formula,  one has
\begin{align*}
\overline{a}(n) = \frac{1}{2 \pi i} \int_C \frac{\overline{A}(q)}{q^{n+1}} \mathrm{d}q,
\end{align*}
where $C:|q|=r<1$. Substitute $q=e^{2 \pi i \tau}$,  with $\tau \in \mathbb{H}$,  to get
\begin{align}\label{integral for cubic}
\overline{a}(n) = \int_{\tau_0}^{\tau_0+1} e^{-2 n \pi i \tau}\overline{A}(e^{2 \pi i \tau}) \mathrm{d}\tau,
\end{align}
where $\tau_0 \in \mathbb{H}$ and as the integrand is regular,  so we can choose any path joining $\tau_0$ to $\tau_0+1$.   We consider our path based on Farey dissections with $\tau_0=i$.   Given any natural number $N$,  we construct Ford circles belonging to Farey sequence of order $N$.   Then for the Ford circle $C(h,k)$,  defined in \eqref{Ford circle C(h,k)},   we choose arc $\gamma_{h,k}$ in such a way that it connects the tangency points $\alpha_1$ and $\alpha_2$,  see \ref{Ford circles}.   Here we denote  $\alpha_1$ and $\alpha_2$  as  $\frac{h}{k} + C'_{h,k}$ and $\frac{h}{k}+C''_{h,k}$, where
\begin{align*}
C'_{h,k}:=  - \frac{k_1}{k(k^2 + k_1^2)}+ i  \frac{1}{k^2 + k_1^2}, \quad \textrm{and} \quad C''_{h,k}:= \frac{k_2}{k(k^2 + k_2^2)} + i \frac{1}{k^2 + k_2^2}.
\end{align*} 
 Note that the arc $\gamma_{h,k}$ is chosen in a way so that it does not touches the real axis and the path of the integration is the collection of all such arcs.

Thus,  considering the above mentioned arcs $\gamma_{h,k}$ corresponding to the fractions $\frac{h}{k}$ in the Farey sequence of order $N$,  the integral \eqref{integral for cubic} takes the form
\begin{align*}
\overline{a}(n)=\sum_{\substack{0\leq h<k \leq N \\ (h,k)=1}} \int_{\gamma_{h,k}} e^{-2 n \pi i \tau}~\overline{A}(e^{2 \pi i \tau}) \mathrm{d}\tau,
\end{align*}
where $\gamma_{h,k}$ is an arc of the Ford circle $\left|\tau-\left( \frac{h}{k} + \frac{i}{2k^2} \right)\right| = \frac{1}{2k^2}$ not touching the $x$-axis. Substituting $\tau=\frac{h}{k} + \zeta$ in the above integral gives
\begin{align*}
\overline{a}(n)=\sum_{\substack{0\leq h<k \leq N \\ (h,k)=1}} e^{-2 n \pi i \frac{h}{k}} \int_{C'_{h,k}}^{C''_{h,k}} e^{-2 n \pi i \zeta}~\overline{A}\left(e^{2 \pi i \left(\frac{h}{k} + \zeta\right)}\right) \mathrm{d}\zeta.
\end{align*}
Replace $\zeta$ by $\frac{iz}{k^2}$ to arrive at
\begin{align}\label{integral before applying transformation formula}
\overline{a}(n)=\sum_{\substack{0\leq h<k \leq N \\ (h,k)=1}} \frac{i}{k^2} e^{-2 n \pi i \frac{h}{k}} \int_{z'_{h,k}}^{z''_{h,k}} e^{\frac{2n\pi z}{k^2}}~ \overline{A}\left(e^{2 \pi i \left(\frac{h}{k} + \frac{iz}{k^2}\right)}\right) \mathrm{d}z,
\end{align}
where 
\begin{align*}
z'_{h,k} =  \frac{k^2}{k^2 + k_1^2} + \frac{ikk_1}{k^2 + k_1^2}, \quad
z''_{h,k} =  \frac{k^2}{k^2 + k_2^2} - \frac{ikk_2}{k^2 + k_2^2},
\end{align*}
lie on circle $|z-\frac{1}{2}|=\frac{1}{2}$.  

Now our goal is to write a transformation formula for the below function, 
\begin{align}\label{A_bar(q)}
\overline{A}\left(e^{2 \pi i \left(\frac{h}{k} + \frac{iz}{k^2}\right)}\right) = \frac{\left\{ f\left(e^{2 \pi i \left(\frac{h}{k} + \frac{iz}{k^2}\right)}\right) \right\}^2 f\left(e^{4 \pi i \left(\frac{h}{k} + \frac{iz}{k^2}\right)}\right)}{f\left(e^{8 \pi i \left(\frac{h}{k} + \frac{iz}{k^2}\right)}\right)}.
\end{align}
Hardy and Ramanujan \cite[Lemma 4.31]{HR1918} obtained the following transformation formula for $f(q)$,  with $hh_1 \equiv -1$ (mod $k$),  
\begin{align}\label{transformation formula for f(q)}
f\left(e^{2\pi i \left(\frac{h}{k} + \frac{iz}{k^2} \right)}\right) = \sqrt{\frac{z}{k}} e^{i\pi s\left(h,k\right)} e^{\frac{\pi}{12k}\left(\frac{k}{z}-\frac{z}{k}\right)}f\left(e^{2\pi i \left(\frac{h_1}{k}+ \frac{i}{z}\right)}\right).
\end{align}
Now employing \eqref{transformation formula for f(q)},  and the transformation formulas for $f(q^2)$ and $f(q^4)$ obtained in \eqref{transformation formula for f(q^2)} and \eqref{transformation formula for f(q^4)} respectively,   in \eqref{A_bar(q)}, we arrive at the following transformation formula for $\overline{A}(q)$:
\begin{align}\label{transformation formula for A(q)}
\overline{A}\left(e^{2 \pi i \left(\frac{h}{k} + \frac{iz}{k^2}\right)}\right) =
\begin{cases}
\frac{z}{k\sqrt{2}} e^{i \pi \left\{2s(h,k) + s(2h,k) -s(4h,k) \right\}} e^{\frac{3 \pi}{16z}} \frac{ \left\{ f\left(e^{2\pi i \left(\frac{h_1}{k}+ \frac{i}{z}\right)}\right) \right\}^2 f\left(e^{2\pi i \left(\frac{h_2}{k}+ \frac{i}{2z}\right)}\right)}{f\left(e^{2\pi i \left(\frac{h_4}{k}+ \frac{i}{4z}\right)}\right)},&\mathrm{if}~ k~\mathrm{odd},\\
\frac{z}{k\sqrt{2}} e^{i \pi \left\{2s(h,k) + s(h,\frac{k}{2})-s(2h,\frac{k}{2}) \right\}} e^{\frac{\pi}{4z}} \frac{ \left\{ f\left(e^{2\pi i \left(\frac{h_1}{k}+ \frac{i}{z}\right)}\right)\right\}^2 f\left(e^{4\pi i \left(\frac{h_3}{k}+ \frac{i}{z}\right)}\right)}{f\left(e^{4\pi i \left(\frac{h_5}{k}+ \frac{i}{2z}\right)}\right)},&\mathrm{if}~ k \equiv 2 (\mathrm{mod}~4),\\
\frac{z}{k} e^{i \pi \left\{2s(h,k) + s(h,\frac{k}{2})-s(h,\frac{k}{4}) \right\}} \frac{ \left\{ f\left(e^{2\pi i \left(\frac{h_1}{k}+ \frac{i}{z}\right)}\right)\right\}^2 f\left(e^{4\pi i \left(\frac{h_3}{k}+ \frac{i}{z}\right)}\right)}{f\left(e^{8\pi i \left(\frac{h_6}{k}+ \frac{i}{z}\right)}\right)},&\mathrm{if}~ k \equiv 0 (\mathrm{mod}~4).
\end{cases}
\end{align}
We now apply the transformation formula \eqref{transformation formula for A(q)} in \eqref{integral before applying transformation formula}. The resulting sum is then partitioned into three components written  as follows:
\begin{align}\label{value of a(n) in terms of S_1, S_2 and S_0}
\overline{a}(n) = S_0 + S_1 + S_2.
\end{align}
In this decomposition,  $S_0$ comprises the contribution coming from $k \equiv 0 (\mathrm{mod}~4)$,  and $S_1$ accounts for the terms with $k$ odd,  $S_2$ includes those with $k \equiv 2 (\mathrm{mod}~4)$.  We start by analyzing $S_1$.  Implementing the transformation formula for $k$ odd in \eqref{integral before applying transformation formula}, we get
\begin{align}\label{S_1}
S_1=\frac{i}{\sqrt{2}}\sum_{\substack{k=1 \\ k~\mathrm{odd}}}^N \frac{A_k^{(1)}(n)}{k^3 }  \int_{z'_{h,k}}^{z''_{h,k}} z e^{\frac{2n\pi z}{k^2} + \frac{3 \pi}{16z}} \frac{ \left\{ f\left(e^{2\pi i \left(\frac{h_1}{k}+ \frac{i}{z}\right)}\right)\right\}^2 f\left(e^{2\pi i \left(\frac{h_2}{k}+ \frac{i}{2z}\right)}\right)}{f\left(e^{2\pi i \left(\frac{h_4}{k}+ \frac{i}{4z}\right)}\right)} \mathrm{d}z,
\end{align}
where 
\begin{align*}
A_k^{(1)}(n)&=\sum_{\substack{h=0 \\ (h,k)=1}}^{k-1}  e^{ \pi i \left(2s(h,k) + s(2h,k)-s(4h,k) \right) -2n \pi i \frac{h}{k} }.
\end{align*}
Our goal now is to simplify the integral present in \eqref{S_1} and denote it as $J_1$.  
For this purpose, we first estimate the bound of few terms present in the integral.  For $\Re(z) \rightarrow 0^{+}$,    $\Re(1/z)$ goes to infinity.  Thus,  one can check that 
\begin{align*}
\left\{ f\left(e^{2\pi i \left(\frac{h_1}{k}+ \frac{i}{z}\right)}\right)\right\}^2 & =  1 +  \mathcal{O}\left( e^{-2\pi \mathrm{Re} \left(\frac{1}{z} \right)} \right),  \\  f\left(e^{2\pi i \left(\frac{h_2}{k}+ \frac{i}{2z}\right)}\right) & =    1 +  \mathcal{O}\left( e^{- \pi \mathrm{Re} \left(\frac{1}{z} \right)} \right),  \\
f\left(e^{2\pi i \left(\frac{h_4}{k}+ \frac{i}{4z}\right)}\right) &=   1 +  \mathcal{O}\left( e^{- \frac{\pi}{2} \mathrm{Re} \left(\frac{1}{z} \right)} \right).  
\end{align*}
Combine these bounds to see that 
\begin{align*}
\frac{ \left\{ f\left(e^{2\pi i \left(\frac{h_1}{k}+ \frac{i}{z}\right)}\right)\right\}^2 f\left(e^{2\pi i \left(\frac{h_2}{k}+ \frac{i}{2z}\right)}\right)}{f\left(e^{2\pi i \left(\frac{h_4}{k}+ \frac{i}{4z}\right)}\right)} = 1 + \mathcal{O}\left( e^{-\frac{\pi}{2}\mathrm{Re} \left(\frac{1}{z} \right)} \right).
\end{align*}
Utilizing the above bound,   $J_1$ becomes
\begin{align*}
J_1&=\int_{z'_{h,k}}^{z''_{h,k}} z e^{\frac{2n\pi z}{k^2} + \frac{3 \pi}{16z}} \left( 1 + \mathcal{O}\left( e^{-\frac{\pi}{2}\mathrm{Re} \left(\frac{1}{z} \right)} \right) \right) \mathrm{d}z \nonumber \\
&= \int_{z'_{h,k}}^{z''_{h,k}} z e^{\frac{2n\pi z}{k^2} + \frac{3 \pi}{16z}} \mathrm{d}z + \mathcal{O}\left( \int_{z'_{h,k}}^{z''_{h,k}} |z| e^{\frac{2n\pi}{k^2} \mathrm{Re}(z)} e^{-\frac{5 \pi}{16}\mathrm{Re} \left(\frac{1}{z} \right)}  |\mathrm{d}z| \right).
\end{align*}
In the first integral above, we add and subtract the integral over the remaining portion of the circle  $|z- \frac{1}{2}|=\frac{1}{2}$,  namely,   the arc from $0$ to $z'_{h,k}$ and $z''_{h,k}$ to $0$.  Since the integrand of $J_1$ is regular,  so to bound the second integral,  we choose the straight line path connecting $z'_{h,k}$ and $z''_{h,k}$. Along this straight line, we have $
|z|\leq \frac{\sqrt{2}k}{N} $ and $ 0< \mathrm{Re}(z)\leq \frac{2k^2}{N^2}. $
Furthermore, the length of the straight line path is  bounded above by $|z'_{h,k}| + |z''_{h,k}|\leq \frac{2\sqrt{2}k}{N}$. This bound combined with the earlier estimates on the integrand allows us to effectively control the contribution of the integral $J_1$  in the subsequent analysis.
Thus, we have
\begin{align}\label{integral_J1}
J_1&= \int_{K^-} z e^{\frac{2n\pi z}{k^2} + \frac{3 \pi}{16z}} \mathrm{d}z - \int_0^{z'_{h,k}} z e^{\frac{2n\pi z}{k^2} + \frac{3 \pi}{16z}} \mathrm{d}z - \int_{z''_{h,k}}^0 z e^{\frac{2n\pi z}{k^2} + \frac{3 \pi}{16z}} \mathrm{d}z + \mathcal{O}\left( \frac{k}{N} e^{\frac{4n\pi}{N^2}} \int_{z'_{h,k}}^{z''_{h,k}} |\mathrm{d}z| \right) \nonumber \\
&= \int_{K^-} z e^{\frac{2n\pi z}{k^2} + \frac{3 \pi}{16z}} \mathrm{d}z - \int_0^{z'_{h,k}} z e^{\frac{2n\pi z}{k^2} + \frac{3 \pi}{16z}} \mathrm{d}z - \int_{z''_{h,k}}^0 z e^{\frac{2n\pi z}{k^2} + \frac{3 \pi}{16z}} \mathrm{d}z + \mathcal{O}\left( \frac{k^2}{N^2}  e^{\frac{4n\pi}{N^2}}    \right),
\end{align}
where $K^-$ denote the circle $|z-\frac{1}{2}|=\frac{1}{2}$ oriented in the clockwise direction.  On the arc from $0$ to $z'_{h,k}$, we have $|z|\leq \frac{\sqrt{2}k}{N},~0 < \mathrm{Re}(z)\leq \frac{2k^2}{N^2},~ \mathrm{Re}\left(\frac{1}{z}\right) = 1$ and the arc length is bounded above by $ \frac{\pi}{2} \frac{\sqrt{2}k}{N}$. Consequently, this gives
\begin{align}\label{bound_from_0_to_z'}
\int_0^{z'_{h,k}} z e^{\frac{2n\pi z}{k^2} + \frac{3 \pi}{16z}} \mathrm{d}z = \mathcal{O}\left( \frac{k^2}{N^2} e^{\frac{4n\pi}{N^2}}     \right).
\end{align}
Analogously, the integral along the segment from $z''_{h,k}$ to $0$ can be estimated by
\begin{align}\label{bound_from_z''_to_0}
\int_{z''_{h,k}}^0 z e^{\frac{2n\pi z}{k^2} + \frac{3 \pi}{16z}} \mathrm{d}z = \mathcal{O}\left( \frac{k^2}{N^2} e^{\frac{4n\pi}{N^2}}    \right).
\end{align}
Finally utilizing the bounds \eqref{bound_from_0_to_z'} and \eqref{bound_from_z''_to_0} in \eqref{integral_J1}, we arrive at
\begin{align}\label{final value of J_1}
J_1 = \int_{K^-} z e^{\frac{2n\pi z}{k^2} + \frac{3 \pi}{16z}} \mathrm{d}z + \mathcal{O}\left( \frac{k^2}{N^2}  e^{\frac{4n\pi}{N^2}}   \right).
\end{align}
Incorporating the expression for $J_1$ obtained in \eqref{final value of J_1} into \eqref{S_1}, we obtain
\begin{align}\label{final value of S_1}
S_1 &=\frac{i}{\sqrt{2}}\sum_{\substack{k=1 \\ k~\mathrm{odd}}}^N \frac{A_k^{(1)}(n)}{k^3 } \int_{K^-} z e^{\frac{2n\pi z}{k^2} + \frac{3 \pi}{16z}} \mathrm{d}z + \mathcal{O}\left(  \frac{e^{\frac{4n\pi}{N^2}}}{N^2}   \sum_{\substack{k=1 \\ k ~\mathrm{odd}}}^N \frac{1}{k} \sum_{\substack{h=0 \\ (h,k)=1}}^{k-1} 1 \right) \nonumber \\
&=\frac{i}{\sqrt{2}}\sum_{\substack{k=1 \\ k~\mathrm{odd}}}^N \frac{A_k^{(1)}(n)}{k^3 } \int_{K^-} z e^{\frac{2n\pi z}{k^2} + \frac{3 \pi}{16z}} \mathrm{d}z + \mathcal{O}\left(    \frac{e^{\frac{4n\pi}{N^2}}  }{N}\right).
\end{align}
In the above estimate, we have employed the trivial bound $|A_k^{(1)}(n)| \leq k$. Proceeding in a similar way, we can obtain analogous estimate for $S_2$ as
\begin{align}\label{final value of S_2}
S_2 &= \frac{i}{\sqrt{2}}\sum_{\substack{k=1 \\ k \equiv 2~ \textup{mod}~4}}^N \frac{A_k^{(2)}(n)}{k^3 } \int_{K^-} z e^{\frac{2n\pi z}{k^2} + \frac{\pi}{4z}} \mathrm{d}z + \mathcal{O}\left(  \frac{e^{\frac{4n\pi}{N^2}}  }{N}  \right),
\end{align}
where
\begin{align*}
A_k^{(2)}(n)&=\sum_{\substack{h=0 \\ (h,k)=1}}^{k-1}  e^{ \pi i \left(2s(h,k) + s\left(h,\frac{k}{2}\right)-s\left(2h,\frac{k}{2}\right) \right)-2n \pi i \frac{h}{k}}. 
\end{align*}
However,   to calculate $S_0$,  we need to use transformation formula \eqref{transformation formula for A(q)} corresponding to $k \equiv 0 (\textrm{mod}~ 4)$.  Note that   there is no term involving $e^{\Re(\frac{1}{z})}$ in the transformation formula and hence the corresponding integral has no dominant contribution.  Therefore,  the evaluation of $S_0$ goes in the error term.  Mainly,  one can show that 
\begin{align} \label{final value of S_0}
S_0 &= \mathcal{O}\left(  \frac{e^{\frac{4n\pi}{N^2}}  }{N}  \right).  
\end{align}
Now, combining \eqref{final value of S_1}, \eqref{final value of S_2} and \eqref{final value of S_0} in  \eqref{value of a(n) in terms of S_1, S_2 and S_0}, we obtain
\begin{align*}
\overline{a}(n) &= \frac{i}{\sqrt{2}}\sum_{\substack{k=1 \\ k~\mathrm{odd}}}^N \frac{A_k^{(1)}(n)}{k^3 } \int_{K^-} z e^{\frac{2n\pi z}{k^2} + \frac{3 \pi}{16z}} \mathrm{d}z + \frac{i}{\sqrt{2}} \sum_{\substack{k=1 \\ k \equiv 2~ \mathrm{mod}~4}}^N \frac{A_k^{(2)}(n)}{k^3 } \int_{K^-} z e^{\frac{2n\pi z}{k^2} + \frac{\pi}{4z}} \mathrm{d}z \nonumber \\
& \hspace{2cm}+ \mathcal{O}\left(  \frac{e^{\frac{4n\pi}{N^2}}  }{N}  \right).   
\end{align*}
Since the left-hand side is independent of $N$, taking the limit as $N\rightarrow \infty$ forces the error term on the right-hand side to vanish. Consequently, we have
\begin{align}\label{final value of a(n)}
\overline{a}(n) = \frac{i}{\sqrt{2}} \sum_{\substack{k=1 \\ k~\mathrm{odd}}}^\infty \frac{A_k^{(1)}(n)}{k^3 } \int_{K^-} z e^{\frac{2n\pi z}{k^2} + \frac{3 \pi}{16z}} \mathrm{d}z + \frac{i}{\sqrt{2}} \sum_{\substack{k=1 \\ k \equiv 2~ \mathrm{mod}~4}}^\infty \frac{A_k^{(2)}(n)}{k^3} \int_{K^-} z e^{\frac{2n\pi z}{k^2} + \frac{\pi}{4z}} \mathrm{d}z.
\end{align}
This provides an exact formula for $\overline{a}(n)$, with the sole remaining task being to verify the convergence of the infinite series. To this end,  note that along the circle $K^{-}$, we have $|z| \leq 1,~ \mathrm{Re}(z) \leq 1$ and $\mathrm{Re}\left( \frac{1}{z}\right)=1$, which leads to
\begin{align*}
\left| \int_{K^-} z e^{\frac{2n\pi z}{k^2} + \frac{3 \pi}{16z}} \mathrm{d}z\right | \leq \pi e^{2n\pi + \frac{3 \pi}{16}}.
\end{align*}
Thus,  we have
\begin{align*}
\left|\frac{i}{\sqrt{2}} \sum_{\substack{k=1 \\ k~\mathrm{odd}}}^\infty \frac{A_k^{(1)}(n)}{k^3} \int_{K^-} z e^{\frac{2n\pi z}{k^2} + \frac{3 \pi}{16z}} \mathrm{d}z \right| \leq \frac{\pi}{\sqrt{2}} e^{2n\pi + \frac{3 \pi}{16}} \sum_{\substack{k=1 \\  k~\mathrm{odd}}}^\infty \frac{1}{k^2}.
\end{align*}
Here again, we have utilized the  bound $|A_k^{(1)}(n)| \leq k$. 
Since the series on the right hand side above is convergent, it follows that the first sum in \eqref{final value of a(n)} is absolutely convergent.   An identical argument works for the convergence of second sum as well. 
Our next objective is to express both integrals in \eqref{final value of a(n)} in terms of classical special functions.  We will illustrate this procedure for the integral in the first sum; the same method can be applied verbatim to the integral in the second sum. To achieve this,   we perform a change of variable $z = \frac{1}{\omega}$, which maps the circular contour $K^{-}$ onto the vertical line $\mathrm{Re}(\omega) = 1$, running from $1 - i\infty$ to $1 + i\infty$. Under this substitution, the integral reduces to
\begin{align*}
\int_{K^-} z e^{\frac{2n\pi z}{k^2} + \frac{3 \pi}{16z}} \mathrm{d}z = - \int_{1-i\infty}^{1+i\infty} \omega^{-3} e^{\left(\frac{2n \pi}{k^2 \omega}+\frac{3\pi \omega}{16}  \right)} \mathrm{d}\omega.
\end{align*}
Next, substitute $\frac{3 \pi \omega}{16} = t$ in the integral on the right-hand side. Under this change of variable,  the real part of the line of integration becomes $\mathrm{Re}(t) = \dfrac{3\pi}{16}$.  Hence,  we get 
\begin{align}\label{line integral involving I Bessel}
\int_{K^-} z e^{\frac{2n\pi z}{k^2} + \frac{3 \pi}{16z}} \mathrm{d}z = -\left( \frac{3 \pi}{16} \right)^2 \int_{\frac{3\pi}{16}-i\infty}^{\frac{3\pi}{16}+i\infty} t^{-3} e^{\left( t +\frac{3n\pi^2}{8k^2t} \right)} \mathrm{d}t.
\end{align}
One can observe that the integral derived above closely resembles the classical integral representation of the modified Bessel function of the first kind,  
namely,  for $c>0$,  $\Re(\nu)>0$, 
\begin{align}\label{integral representation of bessel function}
I_\nu(z)=\frac{(z/2)^\nu}{2\pi i}\int_{c-i \infty}^{c+i \infty} t^{-\nu-1} e^{\left( t + \frac{z^2}{4t}\right)} \mathrm{d}t.
\end{align}
Now invoking \eqref{integral representation of bessel function},  one can simplify the integral in \eqref{line integral involving I Bessel}  in terms of $I_\nu(z)$ as 
\begin{align}\label{integral representation of first integral}
\int_{K^-} z e^{\frac{2n\pi z}{k^2} + \frac{3 \pi}{16z}} \mathrm{d}z = -\frac{3\pi ik^2}{16n} I_2\left( \frac{\pi}{k}\sqrt{\frac{3n}{2}} \right).
\end{align}
Similarly,   the integral in the second sum of \eqref{final value of a(n)} can be simplified in terms of the modified Bessel function $I_\nu(z)$,  namely,  
\begin{align}\label{integral representation of second integral}
\int_{K^-} z e^{\frac{2n\pi z}{k^2} + \frac{\pi}{4z}} \mathrm{d}z = -\frac{\pi ik^2}{4n} I_2\left( \frac{\pi}{k}\sqrt{2n} \right).
\end{align}
Substituting \eqref{integral representation of first integral} and \eqref{integral representation of second integral} into \eqref{final value of a(n)} establishes the desired Rademacher-type exact formula for cubic overpartition function $\overline{a}(n)$. This completes the proof of Theorem \ref{exact formula for overcubic partition theorem}.
\end{proof}
\begin{proof}[Corollary \ref{corollary of asymptotic of cubic overpartition}][]
We start by separating the term corresponding to $k=1$ from the exact formula \eqref{exact formula for overcubic partition equation} for the cubic overpartition, which gives 
\begin{align}\label{split of a(n)}
\overline{a}(n) = M_{\overline{a}}(n) + E_{\overline{a}}(n),
\end{align}
where 
\begin{align}\label{Main_term_overcubic}
M_{\overline{a}}(n) = \frac{3\pi}{16n\sqrt{2}} I_2\left(\pi \sqrt{\frac{3n}{2}} \right),
\end{align}
and 
\begin{align*}
E_{\overline{a}}(n) = \frac{3\pi}{16n \sqrt{2}} \sum_{\substack{k=3 \\ k~\text{odd}}}^\infty \frac{A_k^{(1)}(n)}{k} I_2\left( \frac{\pi}{k} \sqrt{\frac{3n}{2}} \right) + \frac{\pi}{4n \sqrt{2}} \sum_{\substack{k=1 \\ k\equiv 2 ~(\textrm{mod}~4)}} ^\infty \frac{A_k^{(2)}(n)}{k} I_2\left( \frac{\pi}{k} \sqrt{2n} \right).
\end{align*}
Our first goal is to obtain an upper bound for $|E_{\overline{a}}(n)|$.  As we can bound $A_k^{(1)}(n)$ and $A_k^{(2)}(n)$ trivially by $k$ for any $k\geq 1$ and $n\geq 0$,  so we have 
\begin{align}\label{after separating k=2 term}
|E_{\overline{a}}(n)| &\leq \frac{3\pi}{16n \sqrt{2}} \sum_{\substack{k=3 \\ k~\text{odd}}}^\infty I_2\left( \frac{\pi}{k} \sqrt{\frac{3n}{2}} \right) + \frac{\pi}{4n \sqrt{2}} \sum_{\substack{k=1 \\ k\equiv 2 ~(\textrm{mod}~4)}} ^\infty  I_2\left( \frac{\pi}{k} \sqrt{2n} \right) \nonumber \\
&\leq \frac{3\pi}{16n \sqrt{2}} \sum_{k=3}^\infty I_2\left( \frac{\pi}{k} \sqrt{\frac{3n}{2}} \right) + \frac{\pi}{4n \sqrt{2}} \sum_{k=6} ^\infty  I_2\left( \frac{\pi}{k} \sqrt{2n} \right) + \frac{\pi}{4n\sqrt{2}} I_2\left(\pi \sqrt{\frac{n}{2}} \right).
\end{align}
Here, in the last step, we have separated the term corresponding to $k=2$. Next we utilize \eqref{bound_for_I2_series} to bound the first two series as 
\begin{align}\label{bound of first I2 series}
\sum_{k=3}^\infty I_2\left( \frac{\pi}{k} \sqrt{\frac{3n}{2}} \right) \leq \frac{8}{\pi} \sqrt{\frac{2}{3n}} I_1\left( \frac{\pi}{2} \sqrt{\frac{3n}{2}} \right),
\end{align}
and 
\begin{align}\label{bound of second I2 series}
\sum_{k=6} ^\infty  I_2\left( \frac{\pi}{k} \sqrt{2n} \right) \leq \frac{50}{\pi \sqrt{2n}}  I_1\left( \frac{\pi}{5} \sqrt{2n} \right).
\end{align}
Thus, substituting \eqref{bound of first I2 series} and \eqref{bound of second I2 series} in \eqref{after separating k=2 term} and further utilizing the exponential upper bound for $I_{\nu}(s) \leq \sqrt{\frac{2}{\pi s}}e^s$, for any $\nu,  s >0$,   we reach at
\begin{align*}
|E_{\overline{a}}(n)| \leq \frac{6^{\frac{1}{4}}}{\pi n^{\frac{7}{4}}} e^{\frac{\pi}{2} \sqrt{\frac{3n}{2}}} + \frac{5^{\frac{5}{2}}}{\pi (2n)^{\frac{7}{4}}} e^{\frac{\pi}{5} \sqrt{2n}} + \frac{1}{2^{\frac{7}{4}}n^{\frac{5}{4}}} e^{\pi \sqrt{\frac{n}{2}}}.
\end{align*}
As $\frac{1}{\sqrt{2}} > \frac{\sqrt{3}}{2 \sqrt{2}} > \frac{\sqrt{2}}{5}$ and exponential is an increasing function,  hence we have
\begin{align}\label{bound for error}
|E_{\overline{a}}(n)| \leq \frac{1}{2^{\frac{7}{4}}} \left( \frac{2\cdot 5^{\frac{5}{2}}}{\pi n^{\frac{7}{4}}} + \frac{1}{n^{\frac{5}{4}}} \right) e^{\pi \sqrt{\frac{n}{2}}},  
\end{align}
while combining the terms, we also utilized the fact that $6^{\frac{1}{4}}  < \frac{5^{\frac{5}{2}}}{2^{\frac{7}{4}}}$.  This completes the proof of \eqref{asymptotic for cubic}.

Finally,  employing the below classical asymptotic expansion for the modified Bessel function of the first kind,  that is,  
\begin{align*}
I_{\nu}(x) \sim \frac{e^{x}}{\sqrt{2\pi x}} \quad \text{as} \quad x \to \infty,  
\end{align*} in \eqref{Main_term_overcubic} and simplifying yields the asymptotic relation \eqref{asymptotic cubic overpartition}.

\end{proof}

\subsection{Tur\'{a}n inequalities for cubic overpartitions} \label{ Turan inequalities for cubic overpartitions}
Here, we  prove various results which  eventually lead us to obtain the second order Tur\'{a}n inequality i.e.,  log-concavity for the cubic overpartition function.  At the end,  we also present higher Tur\'{a}n inequalities for the cubic overpartition function.  

The first result in this direction provides explicit upper and lower bounds for $\overline{a}(n)$. 
\begin{lemma}\label{bounds for overcubic in terms of main term}
For all $n \geq 393$,  we have
\begin{align*}
M_{\overline{a}}(n) \left(1-\frac{1}{n^6} \right) \leq \overline{a}(n) \leq M_{\overline{a}}(n) \left(1+\frac{1}{n^6} \right),
\end{align*}
where 
\begin{align}\label{Definition of main term}
M_{\overline{a}}(n) = \frac{3\pi}{16n\sqrt{2}} I_2\left(\pi \sqrt{\frac{3n}{2}} \right).
\end{align}
\end{lemma}
\begin{proof}
First,  let us define 
\begin{align}\label{Definition Ga(n)}
G_{\overline{a}}(n):= \frac{1}{M_{\overline{a}}(n)} \left( \frac{1}{2^{\frac{7}{4}}} \left( \frac{2\cdot 5^{\frac{5}{2}}}{\pi n^{\frac{7}{4}}} + \frac{1}{n^{\frac{5}{4}}} \right) e^{\pi \sqrt{\frac{n}{2}}} \right).  
\end{align}
Then,  from \eqref{split of a(n)} and \eqref{bound for error},   we see that
\begin{align*}
|\overline{a}(n) - M_{\overline{a}}(n)| = |E_{\overline{a}}(n)| \leq G_{\overline{a}}(n) M_{\overline{a}}(n).
\end{align*}
Now,  our main goal is to prove that,  for $n \geq 393$,  
\begin{align}\label{Ga(n) bound we need}
G_{\overline{a}}(n) \leq \frac{1}{n^6}.
\end{align}
To get this upper bound for $G_{\overline{a}}(n)$, we first need a lower bound for $M_{\overline{a}}(n)$.  This can be achieved by using a lower bound \eqref{lower bound for I2} for $I_2(s)$.  Thus applying \eqref{lower bound for I2} in \eqref{Definition Ga(n)},   for $n \geq 67$,  one can see that 
\begin{align*}
G_{\overline{a}}(n) \leq 0.027 e^{\pi \sqrt{n}\left(\frac{1-\sqrt{3}}{\sqrt{2}} \right)}.
\end{align*}
Hence proving \eqref{Ga(n) bound we need} is equivalent to verifying the following inequality
\begin{align*}
0.027 e^{\pi \sqrt{n}\left(\frac{1-\sqrt{3}}{\sqrt{2}} \right)} \leq \frac{1}{n^6}, \quad \text{for}~\text{all}~ n \geq 393.
\end{align*}
To do so,  we define the function
\begin{align*}
H(x) := 0.027 e^{\pi \sqrt{x}\left(\frac{1-\sqrt{3}}{\sqrt{2}} \right)} x^6.
\end{align*}
One can check that,  for all $x \geq 54.4516$, $H'(x) <  0$. This shows that $H(x)$ is decreasing on $(54.4516, \infty)$. As $H(393)<1$,  so $H(n)< H(393)<1$ for all $n \geq 393$. This proves that  
\begin{align*}
G_{\overline{a}}(n) \leq 0.027 e^{\pi \sqrt{n}\left(\frac{1-\sqrt{3}}{\sqrt{2}} \right)} \leq \frac{1}{n^6},
\end{align*}
for all $n \geq 393$,  which confirms the validity of \eqref{Ga(n) bound we need} and hence the proof of Lemma \ref{bounds for overcubic in terms of main term} is now complete.
\end{proof}
\begin{remark}\label{remark}
 We would also like to point out that the Lemma \ref{bounds for overcubic in terms of main term} can also be true for any $\alpha >0$ such that 
\begin{align*}
M_{\overline{a}}(n) \left(1-\frac{1}{n^\alpha} \right) \leq \overline{a}(n) \leq M_{\overline{a}}(n) \left(1+\frac{1}{n^\alpha} \right),  \quad \textrm{for}~ n \geq n_0(\alpha).  
\end{align*}

\end{remark}
In the next result, we establish bounds for the product of modified Bessel functions of the first kind.
\begin{lemma}\label{bounds for product of bessel functions}
For $n \geq 2363$, the following bounds are true:
\begin{align*}
\frac{v}{\sqrt{v^+ v^-}} \Xi_{\overline{a}}(n) \leq \frac{I_2(v^-) I_2(v^+)}{(I_2(v))^2} \leq \frac{v}{\sqrt{v^+ v^-}} \Psi_{\overline{a}}(n),
\end{align*}
where 
\begin{align}
\Xi_{\overline{a}}(n) & := \left(1- \frac{9\pi^4}{16 v^3}- \frac{27 \pi^6}{8 v^5} - \frac{405 \pi^8}{2048 v^7}\right) \left( 1-\frac{309}{v^5}-\frac{535}{v^6}\right),  \label{definition Xi}  \\
\Psi_{\overline{a}}(n) & := \left(1-\frac{9\pi^4}{16 v^3}+\frac{81 \pi^8}{256 v^6}\right) \left( 1-\frac{308}{v^5}-\frac{286}{v^6} \right),   \label{definition Psi}
\end{align}
and $v,   v^{+},  v^{-}$ are defined as in \eqref{Definition of v}.  
\end{lemma}
\begin{proof}
First,  we recall
\begin{align}\label{Defn of v, v+,v-}
v = \pi \sqrt{\frac{3n}{2}},  \quad v^+ = \pi \sqrt{\frac{3(n+1)}{2}},  \quad \text{and} \quad v^- = \pi \sqrt{\frac{3(n-1)}{2}}. 
\end{align}
To attain the required bounds for $ \frac{I_2(v^-) I_2(v^+)}{(I_2(v))^2} $, we utilize an upper and lower bound given in \eqref{bound_I2}  for $I_2(s)$.   Mainly,   for all $n \geq 44$,  we get
\begin{align*}
\frac{v}{\sqrt{v^+ v^-}} e^{v^+ + v^- -2v}   U_{\overline{a}}(n) \leq  \frac{I_2(v^-) I_2(v^+)}{(I_2(v))^2} \leq \frac{v}{\sqrt{v^+ v^-}} e^{v^+ + v^- -2v}   V_{\overline{a}}(n),
\end{align*}
where
\begin{align}
U_{\overline{a}}(n) & := \frac{\left( E_{I_2}(v^-) - \frac{31}{v^{-^6}} \right) \left(E_{I_2}(v^+) - \frac{31}{v^{+^6}} \right)}{\left(E_{I_2}(v) + \frac{31}{v^6} \right)^2 },  \label{Definition U}  \\
V_{\overline{a}}(n) & := \frac{\left( E_{I_2}(v^-) + \frac{31}{v^{-^6}} \right) \left(E_{I_2}(v^+) + \frac{31}{v^{+^6}} \right)}{\left(E_{I_2}(v) - \frac{31}{v^6} \right)^2 },   \label{Definition V}
\end{align}
and $E_{I_2}(s)$ is defined as in \eqref{Def_EI2}.
We start by deriving an upper and lower bound for $e^{v^+ + v^- -2v}$. 
In this context,  we first derive bounds for $v^{+},  v^-$ in terms of $v$. 
One can easily check that,  for $n >1$,  
\begin{align*}
\sqrt{n} \left(1-\frac{1}{2n}-\frac{1}{8n^2}-\frac{1}{16n^3}-\frac{1}{n^3} \right) < \sqrt{n-1} < \sqrt{n} \left(1-\frac{1}{2n}-\frac{1}{8n^2}-\frac{1}{16n^3} \right), 
\end{align*}
and then multiply throughout by $\pi \sqrt{\frac{3}{2}}$ to see that
\begin{align}\label{bound for v-}
C_v \leq v^- \leq D_v,
\end{align}
where
\begin{align*}
C_v &:= v-\frac{3\pi^2}{4v}-\frac{9\pi^4}{32v^3}- \frac{27 \pi^6}{128 v^5} - \frac{27 \pi^6}{8 v^5}, \\
D_v &:= v-\frac{3\pi^2}{4v}-\frac{9\pi^4}{32v^3}- \frac{27 \pi^6}{128 v^5}.
\end{align*}
Similarly,  one can verify that
\begin{align}\label{bound for v+}
A_v \leq v^+ \leq B_v,
\end{align} 
where
\begin{align*}
A_v &:= v+\frac{3\pi^2}{4v}-\frac{9\pi^4}{32v^3}+ \frac{27 \pi^6}{128 v^5} - \frac{405 \pi^8}{2048 v^7}, \\
B_v &:=  v+\frac{3\pi^2}{4v}-\frac{9\pi^4}{32v^3}+ \frac{27 \pi^6}{128 v^5}.
\end{align*}
Now combining \eqref{bound for v-} and \eqref{bound for v+},  we obtain
\begin{align*}
-\frac{9\pi^4}{16 v^3}- \frac{27 \pi^6}{8 v^5} - \frac{405 \pi^8}{2048 v^7}  \leq v^+ + v^- - 2v \leq  -\frac{9\pi^4}{16 v^3}.
\end{align*}
Next, we make use of the fact that for $s<0$, $ 1+s < e^s < 1+s+s^2$, which gives 
\begin{align*}
1- \frac{9\pi^4}{16 v^3}- \frac{27 \pi^6}{8 v^5} - \frac{405 \pi^8}{2048 v^7}  \leq e^{v^+ + v^- -2v} \leq 1-\frac{9\pi^4}{16 v^3}+\frac{81 \pi^8}{256 v^6}.
\end{align*}
Our next aim is to look for a lower bound of $U_{\overline{a}}(n)$ and an upper bound for $V_{\overline{a}}(n)$. 
First,   we write 
\begin{align*}
U_{\overline{a}}(n) \geq \frac{\Theta}{\left(E_{I_2}(v) + \frac{31}{v^6} \right)^2}, \quad  V_{\overline{a}}(n) \leq \frac{\Delta}{\left(E_{I_2}(v) - \frac{31}{v^6} \right)^2},
\end{align*}
where $\Theta$ and $\Delta$ can be obtained by first expanding the numerators of $U_{\overline{a}}(n)$ and $V_{\overline{a}}(n)$ in \eqref{Definition U},  \eqref{Definition V},  and  utilizing the bounds \eqref{bound for v-} and \eqref{bound for v+} for $v^-$ and $v^+$,  respectively.  The exact expressions for $\Theta$ and $\Delta$ will be  clear from the below calculations of the numerators of $U_{\overline{a}}(n)$ and $V_{\overline{a}}(n)$.  Utilizing the definition \eqref{Def_EI2} of $E_{I_2}(s)$,   one can see that 
{\allowdisplaybreaks
\begin{align*}
&\left( E_{I_2}(v^-) - \frac{31}{v^{-^6}} \right) \left(E_{I_2}(v^+) - \frac{31}{v^{+^6}} \right) \nonumber \\
&= \frac{1}{v^{-^6}v^{+^6}} \left( v^{-^6}  - \frac{15}{8}v^{-^5}+ \frac{105}{128}v^{-^4} + \frac{315}{1024}v^{-^3} + \frac{10395}{32768}v^{-^2} + \frac{135135}{262144}v^- - 31 \right) \nonumber \\ 
&\hspace{5mm}\times \left( v^{+^6}  - \frac{15}{8}v^{+^5}+ \frac{105}{128}v^{+^4} + \frac{315}{1024}v^{+^3} + \frac{10395}{32768}v^{+^2} + \frac{135135}{262144}v^+ - 31 \right) \nonumber \\
& \geq \frac{1}{v^{-^6}v^{+^6}} \left( v^{-^6}  - \frac{15}{8}v^{-^4}D_v+ \frac{105}{128}v^{-^4} + \frac{315}{1024}v^{-^2}C_v + \frac{10395}{32768}v^{-^2} + \frac{135135}{262144}C_v - 31 \right) \nonumber \\ 
&\hspace{5mm}\times \left( v^{+^6}  - \frac{15}{8}v^{+^4}B_v+ \frac{105}{128}v^{+^4} + \frac{315}{1024}v^{+^2}A_v + \frac{10395}{32768}v^{+^2} + \frac{135135}{262144}A_v - 31 \right):= \Theta,   \\
&\left( E_{I_2}(v^-) + \frac{31}{v^{-^6}} \right) \left(E_{I_2}(v^+) + \frac{31}{v^{+^6}} \right) \nonumber \\
& = \frac{1}{v^{-^6}v^{+^6}} \left( v^{-^6}  - \frac{15}{8}v^{-^5}+ \frac{105}{128}v^{-^4} + \frac{315}{1024}v^{-^3} + \frac{10395}{32768}v^{-^2} + \frac{135135}{262144}v^- + 31 \right) \nonumber \\ 
&\hspace{5mm}\times \left( v^{+^6}  - \frac{15}{8}v^{+^5}+ \frac{105}{128}v^{+^4} + \frac{315}{1024}v^{+^3} + \frac{10395}{32768}v^{+^2} + \frac{135135}{262144}v^+ + 31 \right) \nonumber \\
& \leq \frac{1}{v^{-^6}v^{+^6}} \left( v^{-^6}  - \frac{15}{8}v^{-^4}C_v+ \frac{105}{128}v^{-^4} + \frac{315}{1024}v^{-^2}D_v + \frac{10395}{32768}v^{-^2} + \frac{135135}{262144}D_v + 31 \right) \nonumber \\ 
&\hspace{5mm}\times \left( v^{+^6}  - \frac{15}{8}v^{+^4}A_v+ \frac{105}{128}v^{+^4} + \frac{315}{1024}v^{+^2}B_v + \frac{10395}{32768}v^{+^2} + \frac{135135}{262144}B_v + 31 \right):= \Delta.  
\end{align*}}
Now we want to show that 
\begin{align*}
\frac{\Theta}{\left(E_{I_2}(v) + \frac{31}{v^6} \right)^2} \geq 1-\frac{309}{v^5}-\frac{535}{v^6}, \quad \frac{\Delta}{\left(E_{I_2}(v) - \frac{31}{v^6} \right)^2} \leq 1-\frac{308}{v^5}-\frac{286}{v^6}, \quad \text{for}~\text{all}~n \geq 2363.
\end{align*}
These are equivalent to verifying that 
\begin{align}
v^{12} \Theta - \left(1-\frac{309}{v^5}-\frac{535}{v^6} \right) \left( v^6 E_{I_2}(v)  + {31} \right)^2 \geq 0,  \quad \text{for}~\text{all}~n \geq 2363, \label{Required inequality for lower bound} \\
v^{12} \Delta - \left(1-\frac{308}{v^5}-\frac{286}{v^6} \right) \left(v^6 E_{I_2}(v) - {31} \right)^2 \leq 0,  \quad \text{for}~\text{all}~n \geq 2363.   \label{Required inequality for upper bound}
\end{align}
Substituting expressions of  $A_v, B_v, C_v, D_v, v^+$ and $v^-$ in $\Theta$ and $\Delta$ and expanding, we write the left hand side of \eqref{Required inequality for lower bound} and \eqref{Required inequality for upper bound} as follows: 
\begin{align}
v^{12} \Theta - \left(1-\frac{309}{v^5}-\frac{535}{v^6} \right) \left( v^6E_{I_2}(v) + {31} \right)^2 & = \frac{\sum_{i=0}^{25}\alpha_i v^i}{v^6 \left(v^4 - \frac{9\pi^4}{4}\right)^3},  \label{theta expansion} \\
v^{12} \Delta - \left(1-\frac{308}{v^5}-\frac{286}{v^6} \right) \left(v^6  E_{I_2}(v) - {31} \right)^2 & = \frac{\sum_{j=0}^{25}\beta_j v^j}{v^6 \left(v^4 - \frac{9\pi^4}{4}\right)^3},  \label{delta expansion}
\end{align}
where $\alpha_i's$ and $\beta_j's$ are some real numbers. One can easily verify that,   for all $n \geq 2$, $0 \leq i \leq 20$ and $0 \leq j \leq 21 $,
\begin{align*}
-|\alpha_i|v^i \geq -|\alpha_{21}| v^{21}, \quad |\beta_j|v^j \leq |\beta_{22}| v^{22}.
\end{align*}
Thus,  we have
\begin{align*}
\sum_{i=0}^{25}\alpha_i v^i &\geq - \sum_{i=0}^{21} |\alpha_i| v^i+\alpha_{22} v^{22} +  \alpha_{23} v^{23} +\alpha_{24} v^{24} +\alpha_{25} v^{25} \nonumber \\
&\geq -22|\alpha_{21}|v^{21} + \alpha_{22} v^{22} +  \alpha_{23} v^{23} +\alpha_{24} v^{24} +\alpha_{25} v^{25} := g_1(v),  \\
\sum_{j=0}^{25}\beta_j v^j &\leq  \sum_{j=0}^{22} |\beta_j| v^j+\beta_{23} v^{23} + \beta_{24} v^{24} + \beta_{25} v^{25} \nonumber \\
& \leq 23|\beta_{22}|v^{22} + \beta_{23} v^{23}+ \beta_{24} v^{24} + \beta_{25} v^{25} := g_2(v),
\end{align*}
where
{\allowdisplaybreaks
\begin{align*}
\alpha_{21}&= -\frac{2677185}{2048}-\frac{545363523 \pi^4}{262144} - \frac{8505 \pi^6}{8192} + \frac{15795 \pi^8}{1024},  \nonumber \\
\alpha_{22}&= \frac{242745}{128} + \frac{31185 \pi^4}{16384},   \nonumber \\
\alpha_{23}&=-\frac{5775}{32}-\frac{14175 \pi^4}{4096},   \\
\alpha_{24}&= -\frac{2991}{4}+\frac{3915\pi^4}{512},   \\
\alpha_{25}&= 309 -\frac{405 \pi^4}{128},   \\
\beta_{22}&=\frac{13095}{16}+ \frac{31185 \pi^4}{16384}-\frac{6075 \pi^6}{512},    \\
\beta_{23}&= \frac{2265}{8}-\frac{14175 \pi^4}{4096} + \frac{405 \pi^6}{64},   \\
\beta_{24}&= -745 + \frac{3915 \pi^4}{512},    \\
\beta_{25}&= 308- \frac{405 \pi^4}{128}.  
\end{align*}}
As the largest real root of $g_1^{\prime}(v)$ is $\approx 39.421$, it follows that $g_1^{\prime}(v) >0$ for all $v \geq 40 $. Consequently, $g_1(v)$ is strictly increasing for $v \geq 40$.   Moreover, as $g_1(41.23)> 0$, we deduce that $g_1(v) >  g_1(41.23) >0$ for all $v \geq 41.23 $. This, in turn, establishes that $g_1(v) > 0$ for all $n \geq 115$. In a similar manner, the largest real root of $g_2^{\prime}(v)$ occurs at $ \approx 179.206$. Since the leading coefficient of $g_2(v)$ is negative, we conclude that $g_2^{\prime}(v) < 0$ whenever $v \geq 180$, and thus $g_2(v)$ is strictly decreasing.   Moreover, as $g_2(187) < 0$, it follows that $g_2(v)$ remains negative for all $v \geq 187$. Consequently, we obtain $g_2(v) <0$ for all $n \geq 2363$.
Utilizing these facts in \eqref{theta expansion} and \eqref{delta expansion} along with the observation $v^4 - \frac{9\pi^4}{4} > 0$ for all $n >1$,  we finish the proof of \eqref{Required inequality for lower bound} and \eqref{Required inequality for upper bound}.  This completes the proof of Lemma \ref{bounds for product of bessel functions}.
\end{proof}

\begin{proof}[Theorem \ref{product of cubic overpartition bound}][]
Using Lemma \ref{bounds for overcubic in terms of main term}, we obtain the inequality
\begin{align}\label{overcubic product starting point}
\Omega(n) \frac{\left(1-\frac{1}{(n+1)^6} \right)\left(1-\frac{1}{(n-1)^6} \right)}{\left(1+\frac{1}{n^6} \right)^2} \leq \frac{\overline{a}(n+1)\overline{a}(n-1)}{(\overline{a}(n))^2} \leq \Omega(n) \frac{\left(1+\frac{1}{(n+1)^6} \right)\left(1+\frac{1}{(n-1)^6} \right)}{\left(1-\frac{1}{n^6} \right)^2}, 
\end{align}
where
\begin{align*}
\Omega(n):= \frac{M_{\overline{a}}(n+1) M_{\overline{a}}(n-1)}{(M_{\overline{a}}(n))^2}.
\end{align*}
Next, by using \eqref{Definition of main term} for   $M_{\overline{a}}(n)$,  we see that 
\begin{align*}
\Omega(n) = \frac{n^2}{(n+1)(n-1)} \frac{I_2\left(\pi \sqrt{\frac{3(n+1)}{2}} \right) I_2\left(\pi \sqrt{\frac{3(n-1)}{2}} \right)}{I_2\left(\pi \sqrt{\frac{3n}{2}} \right)^2}.
\end{align*}
Rewriting the above expression in terms of $v,  v^+,  v^-$ \eqref{Defn of v, v+,v-},   we obtain
\begin{align*}
\Omega(n) = \frac{v^4}{v^{+^2}v^{-^2}} \frac{I_2(v^+)I_2(v^-)}{(I_2(v))^2}.
\end{align*} 
Applying the bounds from Lemma \ref{bounds for product of bessel functions}, we deduce that,  for all $n \geq 2363$,
\begin{align*}
\frac{v^5}{(v^+ v^-)^{\frac{5}{2}}} \Xi_{\overline{a}(n)} \leq \Omega(n) \leq \frac{v^5}{(v^+ v^-)^{\frac{5}{2}}} \Psi_{\overline{a}(n)},
\end{align*}
where $\Xi_{\overline{a}(n)}$ and $\Psi_{\overline{a}(n)}$ are defined as in \eqref{definition Xi} and \eqref{definition Psi}, respectively. Moreover, for all $n \geq 7$, the factor $\frac{v^5}{(v^+ v^-)^{\frac{5}{2}}}$ can be bounded as
\begin{align*}
1+ \frac{45 \pi^4}{16 v^4} + \frac{3645 \pi^8}{512 v^8}   \leq \frac{v^5}{(v^+ v^-)^{\frac{5}{2}}} \leq 1+ \frac{45 \pi^4}{16 v^4} + \frac{8 \pi^8}{ v^8}. 
\end{align*}
Consequently, for all $n \geq 2363$,  we have
\begin{align*}
\left(1+ \frac{45 \pi^4}{16 v^4} + \frac{3645 \pi^8}{512 v^8} \right) \Xi_{\overline{a}(n)}  \leq \Omega(n)   \leq \Psi_{\overline{a}(n)} \left( 1+ \frac{45 \pi^4}{16 v^4} + \frac{8 \pi^8}{ v^8} \right).
\end{align*}
Now employing \eqref{definition Xi},  \eqref{definition Psi} and after further simplification,  it yields
\begin{align}\label{bounds of omega}
 1-\frac{9\pi^4}{16 v^3} + \frac{45 \pi^4}{16v^4}-\frac{309}{v^5}-\frac{27 \pi^6}{8v^5} - \frac{535}{v^6} - \frac{\pi^8}{ v^6} \leq \Omega(n)   \leq 1-\frac{9\pi^4}{16 v^3} + \frac{45 \pi^4}{16v^4}-\frac{308}{v^5} - \frac{286}{v^6} + \frac{81\pi^8}{256 v^6}.
\end{align}
To complete the proof, it remains to estimate the multiplicative corrections arising from \eqref{overcubic product starting point},  namely,  
\begin{align*}
 \frac{\left(1-\frac{1}{(n+1)^6} \right)\left(1-\frac{1}{(n-1)^6} \right)}{\left(1+\frac{1}{n^6} \right)^2}, \quad \text{and} \quad \frac{\left(1+\frac{1}{(n+1)^6} \right)\left(1+\frac{1}{(n-1)^6} \right)}{\left(1-\frac{1}{n^6} \right)^2}.
\end{align*}
Expressing these in terms of $v,  v^+,  v^-$ and we require a lower bound for 
\begin{align*}
\frac{\left(1-\frac{729 \pi^{12}}{64v^{+^{12}}} \right)\left(1-\frac{729 \pi^{12}}{64v^{-^{12}}} \right)}{\left(1+\frac{729 \pi^{12}}{64v^{12}} \right)^2},
\end{align*}
and an upper bound for 
\begin{align*}
\frac{\left(1+\frac{729 \pi^{12}}{64v^{+^{12}}} \right)\left(1+\frac{729 \pi^{12}}{64v^{-^{12}}} \right)}{\left(1-\frac{729 \pi^{12}}{64v^{12}} \right)^2}.
\end{align*}
Using series expansions for the denominators of the above two expressions and the bounds \eqref{bound for v-} and \eqref{bound for v+} for $v^-,v^+$,  one obtains the following estimates valid for all $n \geq 1$,
\begin{align}\label{lower bound for power 6 terms}
\frac{\left(1-\frac{729 \pi^{12}}{64v^{+^{12}}} \right)\left(1-\frac{729 \pi^{12}}{64v^{-^{12}}} \right)}{\left(1+\frac{729 \pi^{12}}{64v^{12}} \right)^2} \geq 1-\frac{729\pi^{12}}{v^6},
\end{align}
\begin{align}\label{upper bound for power 6 terms}
\frac{\left(1+\frac{729 \pi^{12}}{64v^{+^{12}}} \right)\left(1+\frac{729 \pi^{12}}{64v^{-^{12}}} \right)}{\left(1-\frac{729 \pi^{12}}{64v^{12}} \right)^2} \leq 1+\frac{729\pi^{12}}{v^6}.
\end{align}
Finally, combining \eqref{bounds of omega}, \eqref{lower bound for power 6 terms} and \eqref{upper bound for power 6 terms} with \eqref{overcubic product starting point}, we arrive at the desired bounds, thereby completing the proof of Theorem \ref{product of cubic overpartition bound}. 
\end{proof}
\begin{proof}[Theorem \ref{log concavity of cubic overpartitions}][]
Utilizing Theorem \ref{product of cubic overpartition bound}, one can observe that
\begin{align}\label{log concavity for a(n)}
1-\frac{\overline{a}(n+1)\overline{a}(n-1)}{(\overline{a}(n))^2} \geq  \frac{9\pi^4}{16 v^3} - \frac{45 \pi^4}{16 v^4} +\frac{308}{v^5} + \frac{286}{v^6} - \frac{81 \pi^8}{256 v^6} - \frac{729 \pi^{12}}{16 v^6},
\end{align}
for $n \geq 2363$. 
Now we consider the polynomial
 \begin{align*}
 g_3(v) = \frac{9\pi^4}{16 }v^3 - \frac{45 \pi^4}{16 }v^2 + 308 v + 286 - \frac{81 \pi^8}{256 } - \frac{729 \pi^{12}}{16 }.
\end{align*}
A direct computation shows that the real roots of $g_3^{\prime}(v)$ are $\approx 0.715853$ and $\approx 2.61748$. Hence, $g_3^{\prime}(v) > 0$ for all $v \geq 3 $, which implies that $g_3(v)$ is increasing for $v \geq 3$.   Since $g_3(93.28) > 0$, it follows that $g_3(v) > 0$ when $v \geq 93.28$.  This concludes that $g_3(v) > 0$ whenever $n \geq 588$. Consequently, we obtain 
\begin{align}\label{bound for g3}
\frac{g_3(v)}{v^6} = \frac{9\pi^4}{16 v^3} - \frac{45 \pi^4}{16 v^4} +\frac{308}{v^5} + \frac{286}{v^6} - \frac{81 \pi^8}{256 v^6} - \frac{729 \pi^{12}}{16 v^6} > 0 
\end{align} 
 for all $n \geq 588$.  Finally,  using \eqref{bound for g3} in \eqref{log concavity for a(n)},  we prove the log concavity for $\overline{a}(n)$ when $n \geq 2363$.  
 Using Mathematica,  one can check that $\overline{a}(n)$ also satisfies log-concavity for $10 \leq n \leq 2363$ and hence $\overline{a}(n)$ attains log concavity for all $n \geq 10$.  This finishes the proof of Theorem \ref{log concavity of cubic overpartitions}.  
\end{proof}

\subsection{Higher order Tur\'{a}n inequalities for cubic overpartitions}

\begin{proof}[Theorem \ref{hyperbolicity of Jensen polynomial}][]
For any $j \in \mathbb{N}\cup \{0\}$,  utilize \eqref{asymptotic cubic overpartition} to see
\begin{align*}
\overline{a}(n+j) \sim \frac{3^{\frac{3}{4}}}{2^{\frac{19}{4}}(n+j)^{\frac{5}{4}}} e^{\pi \sqrt{\frac{3(n+j)}{2}}}, \quad \text{as} \quad  n \to \infty.
\end{align*}
This gives
\begin{align*}
\frac{\overline{a}(n+j)}{\overline{a}(n)} \sim \frac{n^{\frac{5}{4}}}{(n+j)^{\frac{5}{4}}} e^{\pi \sqrt{\frac{3(n+j)}{2}} - \pi \sqrt{\frac{3n}{2}}}.
\end{align*}
Thus,  taking $\log$ on both sides,   we have
\begin{align*}
\log \left(\frac{\overline{a}(n+j)}{\overline{a}(n)}\right) &\sim  \pi \sqrt{\frac{3}{2}} \left( \sqrt{n+j} - \sqrt{n} \right) - \frac{5}{4} \log\left( \frac{n+j}{n} \right).
\end{align*}
Further,  one can simplify the above right hand side expression as follows:
\begin{align*}
 \pi \sqrt{\frac{3}{2}} \left( \sqrt{n+j} - \sqrt{n} \right) - \frac{5}{4} \log\left( \frac{n+j}{n} \right) 
& =  \pi \sqrt{\frac{3}{2}} \sum_{i=1}^\infty {1/2 \choose i}\frac{j^i}{n^{i-1/2}} +  \frac{5}{4} \sum_{i=1}^\infty \frac{(-1)^i j^i}{in^i}  \nonumber \\
&=  \left(\frac{\pi}{2} \sqrt{\frac{3}{2n}}  -  \frac{5}{4n} \right)j - \left(\frac{\pi}{8} \sqrt{\frac{3}{2n^3}}-\frac{5}{8n^2} \right)j^2 \nonumber \\
& + \sum_{i=3}^\infty \left( \frac{\pi}{n^{i-1/2}} \sqrt{\frac{3}{2}}{1/2 \choose i} + \frac{5(-1)^i}{4in^i} \right)j^i.
\end{align*}
Now we consider $A(n)=\frac{\pi}{2} \sqrt{\frac{3}{2n}}-\frac{5}{4n},~ \delta(n)= \sqrt{\frac{\pi}{8} \sqrt{\frac{3}{2n^3}}-\frac{5}{8n^2}}$, and for $i \geq3$,   $$g_i(n)=\frac{\pi}{n^{i-1/2}} \sqrt{\frac{3}{2}}{1/2 \choose i} + \frac{5(-1)^i}{4in^i}.$$
Note that for $3\leq i \leq d$,  one can easily show that 
\begin{align*}
\lim_{n \rightarrow \infty}\frac{g_i(n)}{(\delta(n))^i} 
&=\lim_{n \rightarrow \infty} \frac{\frac{\pi}{n^{i-1/2}} \sqrt{\frac{3}{2}}{1/2 \choose i} + \frac{5(-1)^i}{4in^i}}{\left(\frac{\pi}{8} \sqrt{\frac{3}{2n^3}}-\frac{5}{8n^2}\right)^{i/2}} =0.  
\end{align*}
Moreover,  we can check that $\lim_{n \rightarrow \infty}\frac{g_i(n)}{(\delta(n))^d} = 0$, for all $i \geq d+1.$ 
Thus,  the sequences $\{\overline{a}(n)\},    \{A(n)\},  \{ \delta(n)\},  \{g_i(n)\}$ satisfy hypotheses of Theorem \ref{theorem of Griffin ono rolen and zagier}.  Hence,  it follows that the Jenson polynomials associated with cubic overpartition function can be expressed in terms of Hermite polynomials which are hyperbolic for sufficiently large $n$. This completes the proof of Theorem \ref{hyperbolicity of Jensen polynomial}.

\end{proof}

\subsection{Log-subadditivity and general log-concavity for cubic overpartitions}
In this subsection, we establish the log-subadditivity and a general log-concavity for the cubic overpartition function $\overline{a}(n)$. Before proving Theorem \ref{subadditivity for overcubic partitions}, we state a lemma that will be useful in deriving the log-subadditivity for $\overline{a}(n)$. In Remark \ref{remark}, we noted that in Lemma \ref{bounds for overcubic in terms of main term}, the term $\frac{1}{n^6}$ can be replaced by $\frac{1}{n^\alpha}$ for any $\alpha > 0$. Here, we present a result corresponding to $\alpha = \tfrac{1}{2}$. For the sake of clarity, we also provide a proof, since it requires a sharper bound for $I_2(s)$.

 \begin{lemma}\label{bounds for overcubic for subadditivity}
For all $n \geq 11$,  we have
\begin{align*}
M_{\overline{a}}(n) \left(1-\frac{1}{\sqrt{n}} \right) \leq \overline{a}(n) \leq M_{\overline{a}}(n) \left(1+\frac{1}{\sqrt{n}} \right),
\end{align*}
where 
\begin{align}\label{Definition of main term}
M_{\overline{a}}(n) = \frac{3\pi}{16n\sqrt{2}} I_2\left(\pi \sqrt{\frac{3n}{2}} \right).
\end{align}
\end{lemma}

\begin{proof}
Following the line of argument used in Lemma \ref{bounds for overcubic in terms of main term}, our main goal here is to prove that, for $n \geq 11$,
\begin{align*}
G_{\overline{a}}(n) \leq \frac{1}{\sqrt{n}},
\end{align*}
where $G_{\overline{a}}(n)$ is defined in \eqref{Definition Ga(n)}.  Now we employ more effective bounds \cite[Equations (4.24)-(4.25)]{cubicturan} for $I_2(s)$,   which asserts that,   for all $s\geq 10$,
\begin{align}
I_2(s) &  > \frac{e^s}{\sqrt{2\pi s}} \left(1 - \frac{2}{s}\right),\label{improved lower bound for I2(s)}\\
I_2(s) &  < \frac{e^s}{\sqrt{2\pi s}} \left(1 - \frac{15}{8s} + \frac{2}{s^2}\right).\label{improved upper bound for I2(s)}
\end{align}
We also make use of an elementary inequality, valid for all $s\geq 10$,
\begin{align}\label{usefull inequality}
\left(1 - \frac{2}{s}\right)\left(1 + \frac{3}{s}\right) = 1 + \frac{s - 6}{s^2} > 1.
\end{align}
Utilizing \eqref{improved lower bound for I2(s)} and \eqref{usefull inequality} in \eqref{Definition Ga(n)}, we obtain,   for all $n \geq 7$, 
\begin{align*}
G_{\overline{a}}(n) \leq 66 e^{\pi \sqrt{n} \left(\frac{1-\sqrt{3}}{\sqrt{2}}\right)}.
\end{align*}
Finally,  applying the same argument as in Lemma \ref{bounds for overcubic in terms of main term},  one can deduce that
\begin{align*}
66 e^{\pi \sqrt{n} \left(\frac{1-\sqrt{3}}{\sqrt{2}}\right)}  \leq \frac{1}{\sqrt{n}},
\end{align*}
for all $n \geq 11$. This completes the proof of Lemma \ref{bounds for overcubic for subadditivity}.
\end{proof}

\begin{proof}[Theorem \ref{subadditivity for overcubic partitions}][]
We begin by obtaining an explicit upper and lower bound for $\overline{a}(n)$. From Lemma \ref{bounds for overcubic for subadditivity}, we have
\begin{align*}
\overline{a}(n) \geq  \frac{3\pi}{16n\sqrt{2}} I_2\left(\pi \sqrt{\frac{3n}{2}} \right)  \left(1-\frac{1}{\sqrt{n}} \right).
\end{align*} 
Substituting the bound \eqref{improved lower bound for I2(s)} of $I_2(s)$ into the above expression, we obtain 
\begin{align}\label{lower bound for a(n) for subadditivity}
\overline{a}(n) \geq \frac{3^{\frac{3}{4}}}{2^{\frac{19}{4}} n^{\frac{5}{4}}} e^{\pi \sqrt{\frac{3n}{2}}} \left(1 - \frac{2}{\pi} \sqrt{\frac{2}{3n}}\right) \left(1-\frac{1}{\sqrt{n}} \right) \geq \frac{3^{\frac{3}{4}}}{2^{\frac{19}{4}} n^{\frac{5}{4}}} e^{\pi \sqrt{\frac{3n}{2}}} \left( 1-\frac{8}{5\sqrt{n}} \right).
\end{align}
Here we have used the fact that $1+ \frac{2}{\pi} \sqrt{\frac{2}{3}} < \frac{8}{5}$.  Similarly,  applying the upper bound \eqref{improved upper bound for I2(s)} for $I_2(s)$ yields,
\begin{align}\label{upper bound for a(n) for subadditivity}
\overline{a}(n) \leq \frac{3^{\frac{3}{4}}}{2^{\frac{19}{4}} n^{\frac{5}{4}}} e^{\pi \sqrt{\frac{3n}{2}}} \left( 1 + \frac{1}{\sqrt{n}} \right).
\end{align}
The inequalities \eqref{lower bound for a(n) for subadditivity}-\eqref{upper bound for a(n) for subadditivity} are true for $n \geq 11$ and we have additionally verified by direct computation in Mathematica that they also hold for $1\leq n \leq 10$.  Let us suppose that $1 \leq m \leq n$ and $n=Bm$ with $B\geq 1$.  Note that $B \in \mathbb{Q}$,  whereas $n, m \in \mathbb{N}$.   Utilizing \eqref{lower bound for a(n) for subadditivity}, we obtain
\begin{align*}
\overline{a}(m)\overline{a}(Bm) \geq \frac{3^{\frac{3}{2}}}{2^{\frac{19}{2}} B^{\frac{5}{4}} m^{\frac{5}{2}}} e^{\pi \sqrt{\frac{3}{2}} \left(\sqrt{m}+\sqrt{Bm} \right)} \left( 1-\frac{8}{5\sqrt{m}} \right) \left( 1-\frac{8}{5\sqrt{Bm}} \right),
\end{align*}
while employing \eqref{upper bound for a(n) for subadditivity}, we deduce
\begin{align*}
\overline{a}(m+Bm) \leq \frac{3^{\frac{3}{4}}}{2^{\frac{19}{4}} (m+Bm)^{\frac{5}{4}}} e^{\pi \sqrt{\frac{3(m+Bm)}{2}}} \left( 1 + \frac{1}{\sqrt{m+Bm}} \right).
\end{align*}
Hence, apart from finitely many exceptional cases, it suffices to determine conditions on $m$ and $B$ such that
\begin{align}\label{required inequality for subadditivity}
e^{\pi \sqrt{\frac{3}{2}} \left(\sqrt{m} + \sqrt{Bm} - \sqrt{m+Bm}\right)} \geq \frac{m^{\frac{5}{4}} 2^{\frac{19}{4}}}{3^{\frac{3}{4}}} \left( \frac{B}{B+1} \right)^{\frac{5}{4}} \frac{\left( 1 + \frac{1}{\sqrt{m+Bm}} \right)}{\left( 1-\frac{8}{5\sqrt{m}} \right) \left( 1-\frac{8}{5\sqrt{Bm}} \right)}.
\end{align}
For notational convenience,  we define
\begin{align*}
L_{\overline{a}}(B)&:= \pi \sqrt{\frac{3}{2}} \left(\sqrt{m} + \sqrt{Bm} - \sqrt{m+Bm} \right),\\
H_{\overline{a}}(B)&:= \left( \frac{B}{B+1} \right)^{\frac{5}{4}} \frac{\left( 1 + \frac{1}{\sqrt{m+Bm}} \right)}{\left( 1-\frac{8}{5\sqrt{m}} \right) \left( 1-\frac{8}{5\sqrt{Bm}} \right)}.
\end{align*}
Taking logarithm on both sides of \eqref{required inequality for subadditivity} yields an equivalent condition,  namely,  
\begin{align}\label{general inequality in L and H}
L_{\overline{a}}(B) \geq  \log\left( \frac{m^{\frac{5}{4}} 2^{\frac{19}{4}}}{3^{\frac{3}{4}}} \right) + \log \left( H_{\overline{a}}(B) \right).
\end{align}
As a function of $B$,  one can check that $L_{\overline{a}}(B)$ is increasing for any $m \geq 1$,  whereas $H_{\overline{a}}(B)$ is decreasing for $m \geq 3$. Consequently,  we have
\begin{align*}
L_{\overline{a}}(B) \geq L_{\overline{a}}(1), \quad \mathrm{and} \quad \log \left( H_{\overline{a}}(1) \right) \geq \log \left( H_{\overline{a}}(B) \right).
\end{align*}
Thus, for all $m \geq 3$,  it suffices to verify
\begin{align}\label{inequality in L and H at 1}
L_{\overline{a}}(1) \geq \log\left( \frac{m^{\frac{5}{4}} 2^{\frac{19}{4}}}{3^{\frac{3}{4}}} \right) + \log \left( H_{\overline{a}}(1) \right).
\end{align}
By direct evaluation of $L_{\overline{a}}(1)$ and $H_{\overline{a}}(1)$, we confirm that \eqref{inequality in L and H at 1} holds for all $m \geq 8$. Therefore, to complete the argument, it remains to check \eqref{general inequality in L and H} for finitely many values $3 \leq m \leq 7$ and for these integers,  let $B_m$ denote the real number satisfying
\begin{align*}
L_{\overline{a}}(B_m) = \log\left( \frac{m^{\frac{5}{4}} 2^{\frac{19}{4}}}{3^{\frac{3}{4}}} \right) + \log \left( H_{\overline{a}}(B_m) \right).
\end{align*}
The corresponding values of $B_m$ for $3 \leq m \leq 7$ are displayed in the below Table \ref{Values of B_m}.

\begin{table}[h]
\caption{Values of $B_m$}\label{Values of B_m}
\renewcommand{\arraystretch}{1}
{\begin{tabular}{|l|l|l|l|l|l|l|}
\hline
$m$ & $3$ & $4$ & $5$ & $6$ & $7$  \\ 
\hline
$B_m$ & $369.385\dots$ & $6.011 \dots$ & $2.548 \dots$ & $1.558 \dots$ & $1.105 \dots$ \\
\hline
\end{tabular}}
\end{table}
By the preceding discussion, if $n=Bm\geq m$ is an integer with $B >B_m$, then \eqref{general inequality in L and H} is satisfied, thereby establishing the theorem in these cases.
Thus, the only remaining possibilities are finitely many exceptional pairs of integers with $ 3 \leq m \leq 7$ and $1 \leq \frac{n}{m} = B \leq B_m$. These cases can be directly verified using Mathematica.  This completes the proof of Theorem \ref{subadditivity for overcubic partitions} for $n,  m \geq 3$.  
Finally,  corresponding to $m=1,2$,  it remain to show that, 
\begin{align}
2 \overline{a}(n) &> \overline{a}(n+1),  \forall~ n \neq 1, 3,  \label{form m=1} \\
4 \overline{a}(n) &> \overline{a}(n+2),  \forall~ n \neq 1.\label{for m=2}
\end{align}
To prove \eqref{form m=1}, we utilize the bounds \eqref{lower bound for a(n) for subadditivity} and \eqref{upper bound for a(n) for subadditivity} for $\overline{a}(n)$,  which is equivalent to show that the following function
\begin{align*}
\log (2) + \pi \sqrt{\frac{3}{2}} \left( \sqrt{n} - \sqrt{n+1} \right) - \frac{5}{4} \log \left( \frac{n}{n+1} \right) - \log \left( 1 + \frac{1}{\sqrt{n+1}} \right) + \log \left( 1 - \frac{8}{5\sqrt{n}} \right)
\end{align*} 
is positive. 
One can check that, the above function is positive for all $n \geq 42$, which proves \eqref{form m=1} in this range.  Moreover,  for $1 \leq n \leq 41$, except for $n=1,3$, we find that \eqref{form m=1} holds via numerical computation on Mathematica. In a similar manner,  one can prove \eqref{for m=2}.  This finishes the proof of Theorem \ref{subadditivity for overcubic partitions}.  
\end{proof}

\begin{proof}[Theorem \ref{general log concavity}][]
It is known that log-concavity is equivalent to strong log-concavity,  readers can see \cite{S88}.  Moreover,  if $\{b(k)\}$  satisfies $b(k)^2 > b(k-1) b(k+1)$ for $k\geq k_0>0$,  
then one can show that
\begin{align*}
b(\ell - i) b(k+i) > b(k) b(\ell)
\end{align*}
holds for all $k_0 \leq k <  \ell$ and $0 < i < \ell-k$. Furthermore, Theorem \ref{log concavity of cubic overpartitions} establishes that the cubic overpartition function $\overline{a}(n)$ satisfies strict  log-concave for $n \geq 10$. Thus, by substituting $k=n-m,~\ell=n+m$ and $i=m$, we obtain
\begin{align}\label{general log concavity inequality}
\overline{a}(n)^2 > \overline{a}(n-m) \overline{a}(n+m)
\end{align}
for all $n > m>1$ and $n-m \geq 10$. Accordingly, it suffices to establish \eqref{general log concavity inequality} in the remaining cases,  i.e.,   $1\leq n-m \leq 9$. To resolve this, we aim to show that
\begin{align}\label{required inequality to prove for general log concavity}
\overline{a}(n)^2 \geq \overline{a}(m+1)^2 > \overline{a}(9)\overline{a}(9+2m) \geq \overline{a}(n-m) \overline{a}(n+m),
\end{align}
holds for all $1\leq n-m \leq 9$ and  $m \geq 37$. The remaining finitely many cases, namely those with $1\leq n-m \leq 9$ and  $m<37$ can be checked numerically.  Since $n \geq m+1$, it follows immediately that
\begin{align*}
\overline{a}(n)^2 \geq \overline{a}(m+1)^2.
\end{align*}
Moreover, as $n-m \leq 9$ and $n+m \leq 9+2m$,  so we have
\begin{align*}
\overline{a}(9)\overline{a}(9+2m) \geq \overline{a}(n-m) \overline{a}(n+m).
\end{align*}
Thus,  both the first and third inequalities in \eqref{required inequality to prove for general log concavity} hold trivially. It remains to establish that
\begin{align}\label{inequality in m}
\overline{a}(m+1)^2 > \overline{a}(9)\overline{a}(9+2m),
\end{align}
for all $m \geq 37$. Taking logarithm on both sides of \eqref{inequality in m},  we see that it is equivalent to showing
\begin{align}\label{inequality in log}
2 \log(\overline{a}(m+1)) - \log(\overline{a}(9)) - \log(\overline{a}(9+2m)) >0,
\end{align}
for all $m \geq 37$. To address this,   we employ a lower and upper bound for $\overline{a}(n)$ derived in \eqref{lower bound for a(n) for subadditivity} and \eqref{upper bound for a(n) for subadditivity}, respectively. Utilizing these estimates together with the fact that $\overline{a}(9)=470$, we obtain that the left-hand side of \eqref{inequality in log} is bounded below by
\begin{align*}
& 2 \log\left(\frac{3^{\frac{3}{4}}}{(m+1)^{\frac{5}{4}} 2^{\frac{19}{4}}}\right) + 2\pi \sqrt{\frac{3(m+1)}{2}} + 2 \log\left(1 - \frac{8}{5\sqrt{m+1}}\right) - \log(470)\\ 
&\quad \quad- \log\left(\frac{3^{\frac{3}{4}}}{(9+2m)^{\frac{5}{4}} 2^{\frac{19}{4}}}\right) - \pi \sqrt{\frac{3(9+2m)}{2}} - \log\left(1 + \frac{1}{\sqrt{9+2m}}\right),
\end{align*}
for all $m \geq 10.$ A straightforward calculation shows that this expression is positive for all $m \geq 40$, thereby proving \eqref{inequality in log} in this range. For the remaining values of $m$,  namely,  for $m=37, 38, 39$,  the inequality \eqref{inequality in m} can be verified numerically using Mathematica.  Thus, \eqref{inequality in m} holds  for all $m \geq 37$, which establishes \eqref{required inequality to prove for general log concavity}. This completes the proof of Theorem \ref{general log concavity}.

\end{proof}

\section{Numerical verification of Theorem \ref{exact formula for overcubic partition theorem}} \label{verification table}
To illustrate the accuracy of our exact formula for $\overline{a}(n)$, we provide a numerical verification by computing explicit values of $\overline{a}(n)$ using the first five terms of both series in the formula \eqref{exact formula for overcubic partition equation} and comparing them with the actual cubic overpartition counts.

\begin{table}[h]
\caption{Verification of Theorem \ref{exact formula for overcubic partition theorem}}\label{Table for cubic overpartition}
\renewcommand{\arraystretch}{1}
{\begin{tabular}{|l|l|l|l|}
\hline
 $n$ &  Exact value  & Value from our result  \\ 
\hline
$22$ & $110012$&$110011.99958$ \\
\hline
$12$&$2020$ &$2020.0026$ \\
\hline
$18$&   $ 24962$  & $24961.9983$ \\
\hline
$87$& $1166034258272$ & $1166034258271.996$  \\
\hline
\end{tabular}}
\end{table}

Further,  we compute  $\overline{a}(100)$  numerically by considering contributions of the first five terms of both series present in \eqref{exact formula for overcubic partition equation}: 
\begin{align*}
13080870093246.877  \\
+ 957724.348\\
- 49.363\\
-0.170\\
+0.203\\
+0.005\\
+0.040\\
+0.001\\
+0.001\\
+0.001\\
\hline \\
=13080871050921.943
\end{align*}
The exact value $ \overline{a}(100) = 13080871050922$.

\section{Concluding remarks}

The main objective of this paper is to apply the Hardy-Ramanujan-Rademacher circle method to derive a Rademacher-type infinite series representation for cubic overpartitions. Mainly, we show that $\overline{a}(n)$ can be expressed as the sum of two absolutely convergent series involving Kloosterman-type sums and Bessel functions,  see Theorem \ref{exact formula for overcubic partition theorem}. Using bounds for Kloosterman-type sums and modified Bessel functions of the first kind, we obtain an effective estimate for the error term, which yields the log-concavity of the cubic overpartition function,  see Theorem \ref{log concavity of cubic overpartitions}. Furthermore, by applying a general result of Griffin, Ono, Rolen, and Zagier, we establish higher-order Turán inequalities for the cubic overpartition function,  see Theorem \ref{hyperbolicity of Jensen polynomial}.  In addition, motivated from the work of Bressenrodt and Ono,  we prove the log-subadditivity of the cubic overpartition function,  see Theorem~\ref{subadditivity for overcubic partitions}.  
 We also deduce generalized log-concavity for cubic overpartitions,  see Theorem \ref{general log concavity},  inspired from the work of DeSalvo and Pak.  
We conclude with a few additional observations that might be of independent interest to the reader.

We would like to point out that the generating function \eqref{generating function of cubic overpartition} for the cubic overpartition function is a weakly holomorphic modular form of weight $-1 $ over $\Gamma_{0}(4)$.  Therefore,  it would be interesting to derive our Theorem \ref{exact formula for overcubic partition theorem} using
a result of Zuckerman \cite{Zucker39} and Bringmann-Ono \cite{BO2011}.

The investigation of the hyperbolicity of Jensen polynomials, particularly those associated with the partition function, has received significant attention in recent years. Desalvo and Pak \cite{DeSalvo2015} established that $J_{p}^{2,n-1}$ is hyperbolic for $n \geq 25$. Subsequently, Chen, Jia and Wang \cite{CJW2019} showed that $J_{p}^{3,n-1}$ is hyperbolic for $n\geq 94$ and conjectured that for each $d>1$, there exist a minimal integer $N_p(d)$ such that $J_{p}^{d,n-1}$ is hyperbolic for all $n \geq N_p(d)$,  which was confirmed by Griffin, Ono, Rolen and Zagier \cite{GORZ19} by showing that $J_{p}^{d,n-1}$ is hyperbolic for every degree $d$ for all sufficiently large $n$. Later, Larson and Wagner \cite{LW2019} established an explicit upper bound for $N_p(d)$. In particular, they proved that $N_p(d) \leq (3d)^{24d}(50d)^{3d^2}$ and further determined the exact values $N_p(4)=206$ and $N_p(5)=381$. More recently, Pandey \cite{Pandey2024} studied the Jensen polynomial associated to plane partition function and established an effective upper bound for $N_{\textrm{PL}}(d)$. 

Based on the numerical evidence from Mathematica,  in case of cubic overpartitions function $\overline{a}(n)$,   we conjecture the following. 
\begin{conjecture}
Let $N_{\overline{a}}(d)$ be the minimum integer such that $J_{\overline{a}}^{d,n-1}$ is hyperbolic for all $n \geq N_{\overline{a}}(d) $, then we have
\begin{table}[h]
\renewcommand{\arraystretch}{1}
{\begin{tabular}{|l|l|l|l|l|l|l|}
\hline
 $d$ &  $3$  & $4$ & $5$ & $6$ & $7$  \\ 
\hline
$N_{\overline{a}}(d)$ & $39$ &$89$  &$172$ &$279$ &$423$ \\
\hline
\end{tabular}}
\end{table}
\end{conjecture}
It would be also an exciting problem to find an effective upper bound for $N_{\overline{a}}(d)$.   

Finally, we consider a natural generalization of the cubic overpartition function. Let $\overline{a}_{\ell,r}(n)$ denote the number of representations of $n$ as a sum of natural numbers, where parts divisible by $\ell$ may appear in $r$ different colors, and the first occurrence of each distinct part may be overlined. Its generating function is given by
\begin{align*}
\sum_{n=0}^{\infty} \overline{a}_{\ell, r}(n) q^n =  \frac{(-q;q)_\infty (-q^{\ell};q^{\ell})_{\infty}^{r-1}}{(q;q)_\infty (q^{\ell};q^{\ell})_\infty^{r-1}}.  
\end{align*}
This framework recovers several known cases.  For example,  when $ (r, \ell)=(2,2)$, we obtain the cubic overpartition function; for $r=1$,  it reduces to the classical overpartition function; and for $\ell=2$ with arbitrary $r$,  one obtains the generalized cubic overpartition function studied by Amdeberhan, Sellers, and Singh \cite{ASS25}.

In a forthcoming work, we plan to investigate Rademacher-type exact formula for $\overline{a}_{\ell,r}(n)$ and explore  log-concavity and higher order Turán inequalities for this broader class of partition function.

{\bf Acknowledgement:} 
The authors wish to thank Prof.  Kathrin Bringmann for giving useful suggestions.  
The first author wishes to thank University Grant Commission (UGC), India, for providing Ph.D. scholarship. The second author's research is funded by the Prime Minister Research Fellowship, Govt. of India, Grant No. 2101705. The last author is grateful to the Anusandhan National Research Foundation (ANRF), India, for giving the Core Research Grant CRG/2023/002122 and MATRICS Grant MTR/2022/000545.  We sincerely thank IIT Indore for providing research-friendly environment.

\end{document}